\documentclass[a4paper,12pt]{article}
\usepackage[toc,page]{appendix}
\usepackage{fancyhdr,lastpage}
\usepackage{amsthm}
\usepackage{amsmath}
\usepackage{color}
\newcommand{\cb}{\color{blue}}

\def\R{{\mathbb{R}}}

\usepackage{amssymb}
\numberwithin{equation}{section}
\newtheorem{theor}{Theorem}[section]
\newtheorem{prop}[theor]{Proposition}
\newtheorem{lem}[theor]{Lemma}

\newtheorem{rem}{Remark}[section]
\newcommand{\eps}{\varepsilon}

\numberwithin{equation}{section}
\usepackage{graphics}
\usepackage{fullpage}
\usepackage{url}
\usepackage{enumerate}
\usepackage{hyperref}
\usepackage{pgf}
\title{Single point gradient blow-up on the boundary 
for a Hamilton-Jacobi equation with $p$-Laplacian diffusion}
\author{Amal Attouchi and Philippe Souplet}
\date{}
\begin{document}
\maketitle
\begin{abstract}
We study the initial-boundary value problem for the Hamilton-Jacobi equation with nonlinear diffusion
$u_t=\Delta_p u+|\nabla u|^q$ in a two-dimensional domain for $q>p>2$. 
It is known that the spatial derivative of solutions may become unbounded in finite time while the solutions themselves remain bounded. 
We show that, for suitably localized and monotone initial~data, the gradient blow-up occurs at a single point of the boundary.
Such a result was known  up to now only in the case of linear diffusion ($p=2$). The analysis in the case $p>2$ is considerably more delicate.
\end{abstract}

\eject

\tableofcontents

\eject

\section{Introduction}

\subsection{Problem and main result}

This article is a contribution to the study of the influence of nonlinear diffusion on the qualitative
properties of equations of Hamilton-Jacobi type and, in particular, on the formation of finite-time singularity.
More specifically, we consider the following problem

\begin{equation}\label{eqprincipale}
\left\{
\begin{array}{lll}
u_t=\Delta_p u+|\nabla u|^q, \qquad&(x,y)\in\Omega,&t>0,\\ 
              \noalign{\vskip 1mm}
u(x,y,t)=\mu  y,   \qquad&(x,y)\in\partial\Omega,&t>0,\\
              \noalign{\vskip 1mm}
u(x,y,0)=u_0(x,y),\qquad&(x,y)\in\Omega,&
\end{array}
\right.
\end{equation}
where $\Delta_p$ denotes the $p$-Laplace operator, $\Delta_p=\nabla\cdot(|\nabla u|^{p-2}\nabla u)$.
Throughout this paper, we assume that $\mu\ge 0$ is a constant and that
\begin{equation}\label{hyppq}
q>p>2.
\end{equation}

For reasons that will appear later, we restrict ourselves to a class of planar domains~$\Omega$ which satisfy certain geometric properties.
We assume that, for some $L_1,L_2>0$,
\begin{eqnarray}
&&\hbox{$\Omega\subset \R^2$ is a smooth bounded domain
of class $C^{2+\epsilon}$ for some $\epsilon\in (0,1)$}; \label{hypOmega0} \\
&&\hbox{$\Omega$ is symmetric with respect to the axis $x=0$}; \label{hypOmega1} \\
&&\hbox{$\Omega\subset\{y>0\}$ and $\Omega$ contains the rectangle $(-L_1,L_1)\times(0,2L_2)$}; \label{hypOmega2} \\
&&\hbox{$\Omega$ is convex in the $x$-direction}. \label{hypOmega3}
\end{eqnarray}
In particular, by (\ref{hypOmega2}), $\partial\Omega$ has a flat part, centered at the origin $(0,0)$.
Note that assumption (\ref{hypOmega3}) is equivalent to the fact that $\Omega\cap\{y=y_0\}$ is a line segment for each~$y_0$.

The initial data $u_0$ is taken in $\mathcal{V}_\mu$, where 
$$\mathcal{V}_\mu:=\left\{ u_0\in C^1(\overline{\Omega}), \ u_0\geq \mu y\ \text{ in }\Omega,\ u_0=\mu y\ \text{ on }\, \partial\Omega\right\}.$$
We shall use the following notation throughout:
$$\Omega_+:=\left\{(x,y)\in\Omega;\ x>0\right\}.$$
For $T>0$, set $Q_T=\Omega\times (0,T)$,
$S_T=\partial \Omega\times (0,T)$ and 
$\partial_PQ_T=S_T\cup (\overline\Omega\times \{0\})$ its parabolic boundary.

Problem (\ref{eqprincipale}) is well posed locally in time (see Section~2 for details), with blow-up alternative in $W^{1,\infty}$ norm.
For brevity, when no confusion arises, the existence time of its maximal solution $u$ will be denoted by
$$T:=T_{max}(u_0)\le \infty.$$

It is known (see \cite{AmalJDE, LauSti}) that global nonexistence, i.e. $T<\infty$, occurs for suitably large initial data
(more generally, for problem (\ref{eqprincipale}) in an $n$-dimensional bounded domain with Dirichlet boundary conditions).
Note that the condition $q>p$ is sharp, since the solutions are global and bounded in $W^{1,\infty}$ if $1<q\le p$ (see \cite{stin, LauSti}).
Since it follows easily from the maximum principle that $u$ itself remains uniformly bounded on $Q_T$, 
global nonexistence can only occur through {\it gradient blow-up}, namely 
$$\sup_{Q_T}|u|<\infty \quad\text{ and }\quad \lim_{t\to T}\|\nabla u(\cdot,t)\|_\infty=\infty.$$
This is different from the usual blow-up, in which the $L^{\infty}$ norm
of the solution tends to infinity as $t\to T_{max}$,
which occurs for equations with zero-order nonlinearities,  
such as $u_t=\Delta_p u+u^q$ (see \cite{fujiohta,galak}).
The study of --$L^{\infty}$ or gradient-- blow-up singularities, in particular their location, time and spatial structure is  very  much 
of interest for the understanding of the  physical problems modelled by such equation, as well as for the mathematical richness that they involve.
The $L^{\infty}$ blow-up for the equation $u_t=\Delta_p u+u^q$ has been extensively studied,
both in the case of linear ($p=2$) and nonlinear ($p>2$) diffusion; see respectively the monographs \cite{superlinear} and \cite{SGKM87}
and the numerous references therein.

As for equation (\ref{eqprincipale}) with $p=q=2$, it is known  as the
(deterministic version of the) Kardar-Parisi-Zhang (KPZ) equation, describing the profile of a 
growing interface in certain physical models (see \cite{KPZ}), where $u$ then represents the height of the interface profile.
The case of $ p=2$ and $q\geq1$  is a more general model which was developed by Krug and Spohn,
aiming at studying the effect of the nonlinear gradient term on  the 
properties of solutions (see \cite{KS}). Our main interest in this paper  is to study the effect of a quasilinear gradient diffusivity
 on the localization of the singularities.

For the case of linear diffusion $p=2$,
various sufficient conditions for gradient blow-up and global existence were
provided and qualitative properties were
investigated, such as: nature of the blow-up set, rate and profile of blow-up,
maximum existence time and continuation after blow-up, boundedness of global
solutions and convergence to
a stationary state. We refer for these to the works \cite{Arri, guo, vaz, Zhan, superlinear, single} and the references therein.

The case $p>2$ is far from being completely understood and fewer works deal with the nonlinear diffusion.
The large time behavior of global solutions in bounded or unbounded domains has been studied in 
\cite{lauren3, stin, LauSti, bartier, barle, LauCPDE}.
Concerning the asymptotic description of singularities, results on the gradient blow-up rate in one space dimension can be found in \cite{amal1D,zhan1, rate}.
On the other hand, in any space dimension, it is known  \cite{AmalJDE} that gradient blow-up can take place only on the boundary,
i.e. 
$$GBUS(u_0)\subset\partial\Omega,$$
where the gradient blow-up set is defined by
$$GBUS(u_0)=\Bigl\{x_0\in \overline\Omega;\ \hbox{for any $\rho>0$,} \sup_{(\overline\Omega\cap B_\rho(x_0))\times (T-\rho,T)} |\nabla u|=\infty\Bigr\}.$$
Moreover, the following upper bounds for the space profile of the singularity were obtained in  \cite{AmalJDE}:
 \begin{equation}\label{BernsteinBound} 
|\nabla u|\leq C\delta^{-\frac{1}{q-p+1}}\ \ \text{and}\ \ u \leq C\delta^{\frac{q-p}{q-p+1}}\ \ \text{in}\ Q_T,
\ \ \text{where}\ \delta(x)={\rm dist}(x,\partial\Omega),
\end{equation} 
and they are sharp in one space dimension  \cite{amal1D}. (Our nondegeneracy lemma \ref{milouar1}
below indicates that they are also sharp in higher dimensions.)

It is easy to see, by considering radially symmetric
solutions with $\Omega$ being a ball,  that $GBUS(u_0)$ can be the whole of $\partial\Omega$.
A natural question is then:  
\begin{eqnarray*}
&\hbox{Can one produce examples (in more than one space dimension)} \\
&\hbox{when $GBUS(u_0)$ is a proper subset of $\partial\Omega$, especially a single point ?} 
\end{eqnarray*}
The goal of this article is to provide an affirmative answer to this question.
Our main result is the following.

\begin{theor}\label{singlegbu} Assume (\ref{hyppq})--(\ref{hypOmega3}).
\smallskip

(i) For any $\rho\in (0,L_1)$, there exists $\mu_0=\mu_0(p,q,\Omega,\rho)>0$ such that, for any $\mu\in (0,\mu_0]$,
there exist initial data $u_0$ in $\mathcal{V}_\mu\cap C^2(\overline\Omega)$ such that
the corresponding solution $u$ of~(\ref{eqprincipale}) enjoys the following properties:
\begin{equation}\label{don0}
\hbox{ $T:=T_{max}(u_0)<\infty$ and $GBUS(u_0)\subset [-\rho, \rho]\times\left\{0\right\}$,}
\end{equation}
\begin{equation}\label{don3}
\hbox{$u(\cdot,t)$ is symmetric with respect to the line $x=0$, for all $t\in (0,T)$,}
\end{equation}  
\begin{equation}\label{don4}
u_x\leq 0 \quad \text{ in $\Omega_+\times(0,T)$,}
 \end{equation}  
 \begin{equation}\label{don1} 
u_y\geq \mu/2 \quad\text{ in } Q_T.
\end{equation} 
(ii) For any such $\mu$ and $u_0$, we have
$$GBUS(u_0)=\{(0,0)\}.$$
\end{theor}

\smallskip

A class of initial data satisfying the requirements of Theorem~\ref{singlegbu} is provided in Lemma~\ref{bonnedonnee} below.
We note that,  in the semilinear case $p=2$, a single-point boundary gradient blow-up result was obtained in \cite{single}. 
Although we follow the same basic 
strategy, the proof here is considerably more complicated.
We point out right away that, in view of property (\ref{don1}), the equation is not degenerate for the solutions under consideration.
However, since the essential goal of this article is to study the effect of nonlinear diffusion 
on gradient blow-up, what is relevant here are the {\it large} values of the gradient in the diffusion operator
(rather than the issues of loss of higher regularity that would arise from the degenerate nature of the equation near the level $\nabla u=0$).
It is an open question whether or not single-point gradient blow-up can still be proved in the case $\mu=0$.
Actually, the lower bound (\ref{don1}) on $|\nabla u|$ is crucially used at various points of the proof, which is already very long and involved,
due to the presence of the nonlinear -- even though nondegenerate -- diffusion term (see the next subsection for more details).

\smallskip

Section~2 is devoted to local in time well-posedness and regularity results.
The proof of Theorem~\ref{singlegbu} will be split into several sections, namely sections 3--7 for assertion (i)
and sections 8--10 for assertion (ii) (the latter uses also section~5).
Finally, in two appendices, we provide the proofs of some regularity properties and a suitable parabolic version of Serrin's corner lemma.
Since the proof is quite long and involved,  for the convenience of readers, we now give an outline
of the main steps of the proof.

\medskip

\subsection{Outline of proof}

For sake of clarity, we have divided the proof into a number of intermediate steps,
each of which being relatively short (one or two pages, say), except for step (f),
which involves long and hard computations.
For the convenience of readers, we outline the structure of the proof.
\smallskip 

(a) {\it Preliminary estimates} (Lemmas~\ref{uclass}--\ref{estimderivtemps}):  
The symmetry in the variable~$x$ and 
the decreasing property for $x>0$
are basic features in order to expect single-point gradient blow-up.
Besides $u$ being bounded, 
we also have boundedness of $u_t$. Moreover, for sufficiently small $\mu$ and under a suitable assumption on $u_0$,
we show that $u_y\ge \mu/2$.
These bounds on $u_t$ and $u_y$ seem necessary in the very long calculations of the key step (f) below.
In turn, the positivity of $u_y$ guarantees that solutions are actually classical and that $D^2u, D^3u$ satisfy some bounds
which seem also necessary to the argument, especially in view of the application of the Hopf lemma and 
the Serrin corner lemma in step (g).

\smallskip
(b) {\it Finite time gradient blow-up for suitably concentrated initial data} (Lemma~\ref{blowupdata}): 
by using a rescaling argument and known blow-up criteria, we show that the solution blows up in finite time provided
the initial data is suitably concentrated in a small ball near the origin.

\smallskip

(c) {\it Local boundary gradient control} (Lemma \ref{localgrad}): if the gradient remains bounded on the boundary near a given boundary point,
then the gradient remains also bounded near that point inside the domain, hence it is not a blow-up point.
This is proved by a local Bernstein type argument. 
\smallskip

(d) {\it Localization of the gradient blow-up set} (Lemma~\ref{localiz}): if an initial data is suitably concentrated near the origin, 
then the gradient blow-up set
is contained in a small neighborhood of the origin. This is proved by constructing comparison functions
which provide a control of the gradient on the boundary outside a small neighborhood of the origin,
and then applying step (c).
One then constructs (Lemma~\ref{bonnedonnee}) initial data which also fulfill the assumptions in (a) and (b).
This ensures the existence of ``well-prepared'' initial data
and thereby completes the proof of assertion (i) of Theorem~\ref{singlegbu}.

\smallskip

(e) {\it Nondegeneracy of gradient blow-up} (Lemma~\ref{milouar1}): 
if the solution is only ``weakly singular'' in a neighborhood of a boundary point,
then the singularity is removable. 
\footnote{In this context, the term ``nondegeneracy'' describes a property of finite-time singularities 
(like in, e.g., Giga and Kohn~\cite{GK89})
and should not be confused with the notion of nondegenerate diffusion mentioned after Theorem~\ref{singlegbu}.}
More precisely,
we show the existence of $m=m(p,q) \in(0,1)$ such that, for a given point $(x_0, 0)$ on the flat part of $\partial\Omega$,
if $u(x,y,t)\le c(x)y^m$ near $(x_0, 0)$ for $t$ close to $T$,
then $(x_0,0)$ is not a gradient blow-up point. In view of step (c), it suffices to control the 
gradient on the boundary near the point~$(x_0, 0)$.
This is achieved by constructing special comparison functions,
taking the form of ``regularizing (in time) barriers''.
\smallskip

(f) {\it Verification of a suitable parabolic inequality for an auxiliary function} $J$, of the form 
$$J(x,y,t)=u_x+ kxy^{-\gamma} u^{\alpha}.$$ 
(Proposition~\ref{PJneg}). This is the most technical step and gives rise to very long computations.
Those computations make use, among many other things,
of the singular, Bernstein-type boundary gradient estimate (\ref{BernsteinBound}), obtained in \cite{AmalJDE}.
They use the bound on $u_t$ and the lower bound on $u_y$, obtained at step (a)
(it is not clear if the latter could be relaxed
\footnote{It seems that the constants in some of the key estimates there are nonuniform as $\mu\to 0^+$,
which prevents us to argue by a limiting procedure from the case $\mu>0$.}).

We note that a similar function $J$ was introduced in \cite{single} to treat the semilinear case $p=2$.
The function in  \cite{single} was a 2D-modification of a one-dimensional device from~\cite{FML}, used there to show single point $L^\infty$ blow-up
for radial solutions of equations of the form $u_t-\Delta_p u=u^q$ (for $p=2$, see also~\cite{galak} for $p>2$).
Although the ideas are related, the calculations here are considerably harder than in \cite{FML, galak, single}.

\smallskip

(g) {\it Verification of initial-boundary conditions for the auxiliary function $J$
in a small subrectangle} near the origin. This requires a delicate
parabolic version of Serrin's corner lemma, which we prove in Appendix~2
(see Proposition~\ref{CornerLemma}).
\smallskip

(h) {\it Derivation of a weakly singular gradient estimate near the origin and conclusion.}
Steps (d) and (e) imply $J\le 0$ by the maximum principle.
By integrating this inequality we obtain an inequality of the form $u(x,y,t)\le c(x)y^k$ as $y\to 0$,
for each small $x\ne 0$ and some $k>m$. In view of step (e), this shows that $(0,0)$ is the only gradient blow-up point.

\smallskip

\section{Local well-posedness and regularity}
In this section we 
 consider the question of local existence and regularity for problem (\ref{eqprincipale}).
 Actually, we consider the slightly more general problem
\begin{eqnarray}
u_t-\Delta_p u&=&|\nabla v|^q \quad\hbox{in $Q_T$}, \label{1a}\\
u&=&g\quad\hbox{on $S_T$}, \label{1b}\\
u(\cdot,0)&=&u_0\quad\hbox{in $\Omega$},\label{1c}
\end{eqnarray}
where the boundary data $g$ and initial data $u_0$ satisfy:
\begin{equation}\label{HypBD}
\hbox{$g\ge 0$ is the trace on $\partial\Omega$ of a regular function in $C^{2+\gamma}(\overline\Omega)$
for some $\gamma\in(0,1)$}
\end{equation}
and
\begin{equation}\label{HypID}
u_0 \in W^{1,\infty}(\Omega),\quad u_0 \ge 0,\quad u_0=g\ \hbox{ on $\partial\Omega$}.
\end{equation}
A function $u$ is called a weak super- (sub-) solution of problem \eqref{1a}--\eqref{1c} on $Q_T$ if
$u(\cdot,0)\geq (\leq) \,u_0$ in $\Omega$, $u\geq  (\leq)\, g$ on $S_T$,
$$u\in C(\overline{\Omega}\times [0,T))\cap L^q(0,T; W^{1,q}(\Omega)),
 \quad u_t\in L^2(0,T; L^2(\Omega))$$
and the integral inequality
$$
 \int\int_{Q_T} u_t \psi+|\nabla u|^{p-2} \nabla u\cdot\nabla \psi \, dx\, dt\geq(\leq)\int\int _{Q_T} |\nabla u|^q \psi \,dx \,dt
$$
holds for all $\psi\in C^0(\overline{Q_T}) \cap L^p(0,T; W^{1,p}(\Omega))$ such that $ \psi\geq 0$ and $\psi=0$ on $S_T$.
A function $u$ is a weak solution of  \eqref{1a}--\eqref{1c}  if it is a super-solution and a sub-solution.

The following result was established in \cite[Theorem~1.1]{AmalJDE} (actually in any space dimension).

\begin{theor}\label{theorsou1} Assume (\ref{hypOmega0}) and $q>p-1>1$. Let $M_1>0$, let $u_0, g$ satisfy \eqref{HypBD}--\eqref{HypID}
and $\|\nabla u_0\|_\infty\le M_1$.
Then:
\begin{enumerate}[(i)]
\item  There exists a time $T _0=T_0(p,q,\Omega,M_1,\|g\|_{C^2})>0$ and a weak solution $u$ of \eqref{1a}--\eqref{1c} on $[0,T_0)$,
which moreover satisfies $u \in L^\infty(0, T_0; W^{1,\infty}(\Omega))$. 
Furthermore, $\nabla u$ is locally H\"older continuous in $Q_{T_0}$.

\item For any $\tau>0$, the problem \eqref{1a}--\eqref{1c} has at most one weak solution $u$ such that
$u \in L^\infty(0, \tau; W^{1,\infty}(\Omega))$. 
\smallskip
\item There exists a (unique) maximal, weak solution of \eqref{1a}--\eqref{1c}
in $L^\infty_{loc}([0, T);$ $W^{1,\infty}(\Omega))$, still denoted by $u$, with existence
time denoted by $T=T_{max}(u_0)$. Then
\begin{equation}\label{minmaxu_0}
\min_{\overline\Omega}u_0\le u\le\max_{\overline\Omega}u_0\quad\hbox{ in $Q_T$},
\end{equation}
$\nabla u$ is locally H\"older continuous in $Q_{T}$ and
$$\hbox{if $T<\infty$, then $\lim_{t\to T} \|\nabla u(t)\|_\infty=\infty$ (gradient blow up, GBU).}$$
\end{enumerate}
\end{theor}

\begin{rem}\label{remsou1} 
We also have a comparison principle for problem \eqref{1a}--\eqref{1c}, cf.~\cite[Proposition 2.1]{AmalJDE}.
More precisely, if $v_1,v_2\in C(\overline Q_T)$ are weak sub-/super-solutions of  \eqref{1a}--\eqref{1c} in $Q_T$, then
$$\sup_{Q_T} \ (v_1-v_2)\leq \sup_{\partial Q_T}\ (v_1-v_2).$$
\end{rem}

As one expects, the solution will possess additional regularity if
we know that $|\nabla u|$ remains bounded away from $0$.
This is made precise by the following result, which is a consequence of regularity theory \cite{lady,lieb}
for quasilinear uniformly parabolic equations.
However, for completeness, we provide a proof in Appendix~1.

\begin{theor}\label{theorsou2} 
Under the assumptions of Theorem \ref{theorsou1}, suppose also that $\inf_{Q_T}|\nabla u|>0$. 

(i) Then $u$ is a classical solution in $Q_T$ and
\begin{equation}\label{uC21}
u\in C^{2+\alpha,1+\alpha/2}_{loc}(\overline\Omega\times (0,T)).
\end{equation}
for some $\alpha\in (0,1)$.

(ii) Moreover,
\begin{equation}\label{nablauC21}
\nabla u\in C^{2+\beta,1+\beta/2}_{loc}(\Omega\times (0,T)).
\end{equation}
for some $\beta\in (0,1)$.

(iii) If the boundary conditions in \eqref{1b} depend only on $y$, then
\begin{equation}\label{uxC21}
u_x\in C^{2+\beta,1+\beta/2}_{loc}(\overline\Omega\times (0,T)).
\end{equation}
for some $\beta\in (0,1)$.
\end{theor}

\section{Preliminary estimates: $x$-symmetry, lower bound on $u_y$ and  bound on~$u_t$}

{\bf Notation.}\ Throughout the paper, we shall use the summation convention on repeated indices,
in expressions of the form $a_{ij}u_{ij}$ or $a_{ij}u_iu_j$.
Also, the letter $C$ will denote positive constants which may vary from line to line,
and whose dependence will be indicated if necessary.

\begin{lem}\label{uclass}
Let $\mu>0$ and $u_0\in\mathcal{V}_\mu$.
Assume
\begin{eqnarray}
&\hbox{$u_0$ is symmetric with respect to the line $x=0$}, \label{don00a} \\
&\hbox{$\partial_x u_0\le 0$ in $\Omega_+$}. \label{don0a}
\end{eqnarray}
Then we have
\begin{equation}\label{compuh0}
u(x,y,t)\geq \mu y \quad\text{in }Q_T.
\end{equation}
and properties (\ref{don3})-(\ref{don4}) are satisfied.
\end{lem}

\begin{proof}
Property (\ref{don3}) is a direct consequence of (\ref{don00a}) and the local-in-time uniqueness.
Due to the assumption $u_0\in\mathcal{V}_\mu$, $\underline{v}=\mu y$ is a subsolution of \eqref{eqprincipale}.
This implies (\ref{compuh0}).

To prove (\ref{don4}), fix $h>0$ and let 
$$u_\pm=u(x\pm h,y,t)\quad\hbox{ for $(x,y)\in \Omega_h:=
\{(x,y)\in \Omega_+;\ (x+h,y)\in \Omega\}$ and $t\in (0,T)$.}$$
Owing to (\ref{hypOmega1}) and (\ref{hypOmega3}), we see that $(x-h,y)\in \Omega$ for all $(x,y)\in \Omega_h$,
so that $u_-$ is well defined.
The functions $u_\pm$ are weak solutions of (\ref{eqprincipale})$_1$ in $\Omega_h\times (0,T)$.
Also $u_+\le u_-$ at $t=0$, due to (\ref{don00a})-(\ref{don0a}) and (\ref{hypOmega3}).

Let $(x,y)\in \partial\Omega_h$. If $x=0$, then 
$u_+(x,y,t)=u(h,y,t)=u(-h,y,t)=u_-(x,y,t)$ by~(\ref{don3}).
If $x>0$, then $(x+h,y)\in \partial\Omega$ as a consequence of $(x,y)\in \partial\Omega_h$ and (\ref{hypOmega3}).
So, by (\ref{compuh0}), we have
$$
u_+(x,y,t)=u(x+h,y,t)=\mu y \le u_-(x-h,y,t).
$$
We deduce from the comparison principle that $u_+\le u_-$ in $\Omega_h\times (0,T)$,
which implies~(\ref{don4}).
\end{proof}

Our next lemma provides a useful supersolution of problem \eqref{eqprincipale}.

\begin{lem}\label{GlobalBarrier}
For $0<\rho<L_1$, denote $\Sigma_\rho=[-\rho/2,\rho/2]\times \{0\}$ and 
$\Sigma'_\rho=\partial\Omega\setminus ((-\rho,\rho)\times \{0\})$.
There exist $\mu_0=\mu_0(p,q,\Omega,\rho)>0$ and a function $U\in C^2(\overline\Omega)$,
depending on $p,q,\rho$, with the following properties:
\begin{eqnarray}
&&U>0\quad\text{ on  $\Omega\cup\Sigma_\rho$}, \label{proppsi1} \\
&&U=0\quad\text{ on  $\Sigma'_\rho$},  \label{proppsi2} \\
&&|U_y|\le 1/2\quad\text{ in  $\Omega$} \label{proppsi3}
\end{eqnarray}
and, for all $0<\mu\le\mu_0$, the function $\bar U=\mu (y+U)$ satisfies
\begin{equation} \label{auxilbarU}
-\Delta_p \bar U\ge |\nabla \bar U|^q
\quad\hbox{in $\Omega$}.
\end{equation}
\end{lem}

\begin{proof}
Fix a nonnegative function $\phi\in C^3(\R^2)$ such that $\phi=1$ on $\Sigma_\rho$ and $\phi=0$ on $\Sigma'_\rho$.
We shall look for $U$ under the form $U=\eps V$, 
$\eps>0$, where $V\in C^2(\overline\Omega)$ is the classical solution of the linear elliptic problem
\begin{equation} \label{auxilvarphi}
\left\{
\begin{array}{lll}
-\bigl[V_{xx}+(p-1)V_{yy}\bigr] =1 \qquad&\hbox{in $\Omega$,}\\
              \noalign{\vskip 1mm}
V=\phi \qquad&\hbox{on $\partial\Omega$.}
\end{array}
\right.
\end{equation}
Note that $V>0$ in $\Omega$ by the strong maximum principle which,
along with the boundary conditions in \eqref{auxilvarphi}, will guarantee \eqref{proppsi1}-\eqref{proppsi2}.
Let $\bar U=\mu (y+U)$. 
Assume $0<\eps<1/(\|V_x\|_\infty+2\|V_y\|_\infty)$, 
which implies \eqref{proppsi3}, as well as $|1+\eps V_y|\ge 1/2$
and $|\nabla \bar U|\le 2\mu$.
To check \eqref{auxilbarU}, we compute:
$$
\begin{aligned}
\Delta_p \bar U
&=|\nabla \bar U|^{p-2}\Bigl[\Delta \bar U+(p-2)\dfrac{\bar U_i\bar U_j\bar U_{ij}}{|\nabla \bar U|^2}\Bigr]\\
&=\mu\eps|\nabla \bar U|^{p-2}\Bigl[V_{xx}+V_{yy}+
(p-2)\dfrac{(\eps V_x)^2V_{xx}+2(\eps V_x)(1+\eps V_y)V_{xy}
+(1+\eps V_y)^2V_{yy}}{(1+\eps V_y)^2+(\eps V_x)^2}\Bigr]\\
&=\mu\eps|\nabla \bar U|^{p-2}\Bigl[V_{xx}+(p-1)V_{yy}+
(p-2)\dfrac{(\eps V_x)^2(V_{xx}-V_{yy})+2(\eps V_x)(1+\eps V_y) V_{xy}}{(1+\eps V_y)^2+(\eps V_x)^2}\Bigr]\\
&\le\mu\eps |\nabla \bar U|^{p-2}\Bigl[-1+C\eps^2+C\eps\Bigr]\\
\end{aligned}
$$
 in $\Omega$, with $C>0$ depending only on $\Omega,\rho, p$ (through $V$).
We may then choose $\eps$ depending only on $\Omega,\rho, p$, such that
$-\Delta_p \bar U\ge \frac{\mu\eps}{2}|\nabla \bar U|^{p-2}.$
Next, since $q>p$, for any $\mu\le \mu_0$ with $\mu_0=\mu_0(p,q,\Omega,\rho)>0$ sufficiently small, we have
$$-\Delta_p \bar U\ge \frac{\mu\eps}{2}|\nabla \bar U|^{p-2}\ge \frac{\mu\eps}{2(2\mu)^{q+2-p}}|\nabla \bar U|^q\ge |\nabla \bar U|^q
\quad\hbox{in $\Omega$}.$$
\end{proof}

Based on Lemma \ref{GlobalBarrier}, we construct a class of solutions such that $u_y$ satisfies a positive lower bound.

\begin{lem}\label{uclass2}
Let $0<\rho<L_1$ and let $\mu_0, \bar U$ be given by Lemma \ref{GlobalBarrier}. 
Assume that $0<\mu\le\mu_0$ and $u_0\in\mathcal{V}_\mu$ satisfy
\begin{equation}\label{don0b2}
u_0(x,y)\le\mu \bigl(y+c \chi_{(-\rho/2,\rho/2)\times (0,L_2)}\bigr) 
\quad\text{in }\Omega,
\end{equation}
with $c=c(p,q,\Omega,\rho)>0$ sufficiently small.

(i) Then $u\le \bar U$ in $Q_T$.

(ii) Assume in addition that 
\begin{equation}\label{donuy}
\partial_y u_0\geq \mu/2 \quad\text{in }\Omega.
\end{equation}
Then
\begin{equation} \label{soluy}
\partial_y u\geq \mu/2\quad\text{in } Q_T.
\end{equation}
\end{lem}

\begin{proof}
(i) Let $\bar U=\mu (y+U)$ be given by Lemma \ref{GlobalBarrier}.
From \eqref{proppsi1}, we know that
$$c:=\displaystyle\min_{[-\rho/2,\rho/2]\times [0,L_2]} U>0.$$
Under assumption \eqref{don0b2}, we thus have $u_0\le \bar U$ in $\Omega$.
Since $u=\mu y\le\bar U$ on $S_T$,
we infer from the comparison principle that $u\le \bar U$ in $Q_T$.

\smallskip
(ii) Set $\delta_0=\mu/2$, fix $h>0$ and let $\tilde\Omega_h=\{(x,y)\in \Omega;\ (x,y+h)\in \Omega\}$. We observe that 
$$u_1:=u(x,y,t)+\delta_0 h\quad\text{and}\quad \text u_2:=u(x, y+h,t)$$
are weak solutions of (\ref{eqprincipale})$_1$ in $\tilde\Omega_h\times (0,T)$.

Let $(x,y)\in \partial\tilde\Omega_h$. If $(x,y)\in \partial\Omega$, then, by (\ref{compuh0}), we have
$$
u_2(x,y,t)=u(x,y+h,t)\ge \mu(y+h) =u(x,y,t)+\mu h\ge u_1(x,y,t).
$$
Otherwise, we have $(x,y)\in \Omega$ and $(x,y+h)\in \partial\Omega$. So there is a minimal $\tilde h\in (0,h]$ such that
$(x,y+\tilde h)\in \partial\Omega$. By the mean-value inequality, it follows that,
for some $\theta\in (0,1)$,
$$U(x,y)=U(x,y+\tilde h)-\tilde h U_y(x,y+\theta \tilde h)\leq |U_y(x,y+\theta \tilde h)|h\le h/2,$$
where we used \eqref{proppsi3}.
Therefore, $\bar U(x,y)\le \mu(y+h/2)$, hence
$$
u(x,y+h,t)-u(x,y,t)\geq \mu(y+h)-\bar U(x,y)\ge \frac{\mu h}{2}.
$$
We have thus proved that
\begin{equation}\label{compuh3}
u_2\ge u_1 \quad\text{on } \partial\tilde\Omega_h.
\end{equation}

On the other hand, using (\ref{donuy}) and the fact that $u_0=\mu y$ on $\partial\Omega$, 
it is not difficult to show that $y\mapsto u_0(x,y)-\delta_0 y$ is nondecreasing in $\Omega$.
(Note that in case $\Omega$ is nonconvex, this is not a mere consequence of (\ref{donuy}) alone).
It follows that
\begin{equation}\label{compuh4}
u(x,y+h,0)\ge u(x,y,0)+\delta_0 h \quad\text{in}\ \tilde\Omega_h.
\end{equation}
Owing to (\ref{compuh3})-(\ref{compuh4}), we may then apply the comparison principle
to deduce that $u_2\ge u_1$ in $\tilde\Omega_h\times (0,T)$. 
Since $h$ is arbitrary, the desired conclusion (\ref{soluy}) follows immediately.
\end{proof}

Assuming that $u_0$ is sufficiently regular, we also get an estimate on the time derivative. 

\begin{lem}\label{estimderivtemps}
Let $\mu\ge 0$ and assume that $u_0\in\mathcal{V}_\mu\cap C^2(\overline{\Omega})$.
Then 
 \begin{equation}
|u_t|\leq  \tilde{C}_1 :=\|\Delta_p u_0+|\nabla u_0|^q\|_\infty \quad\text{in }Q_T.  
\end{equation}
\end{lem}

\begin{proof}
 It is easy to see that $v_\pm(x,y,t):=u_0(x,y)\pm\tilde{C}_1t$ are respectively super- and sub-solution of 
\eqref{eqprincipale} in $Q_T$. The comparison principle implies that
\begin{equation} \label{compu0}
u_0(x,y)-\tilde{C}_1 t\leq u(x,y,t)\leq u_0(x,y)+\tilde{C}_1 t\quad\text{in } Q_T.
\end{equation}
Now fix $h\in (0,T)$ and set $w_\pm(x,y,t):=u(x,y,t+h)\pm\tilde{C}_1h$.
By (\ref{compu0}), we have $w_-(x,y,0)\le u_0(x,y)\le w_+(x,y,0)$ and it follows that 
$w_\pm$ are respectively super- and sub-solution of~\eqref{eqprincipale} in $Q_{T-h}$.
By a further application of the comparison principle, we deduce that
$$|u(x,y,t+h)-u(x,y,t)|\le \tilde{C}_1 h\quad\text{in } Q_{T-h}.$$
Since $h$ is arbitrary, we conclude by dividing by $h$ and sending $h\to 0$.
\end{proof}

\section{Finite-time gradient  blow-up for concentrated initial data}

In this section, by a rescaling argument, we show that the solution of \eqref{eqprincipale} blows up in finite time provided
the initial data is suitably concentrated in a small ball near the origin.
For such concentrated initial data, under some additional assumptions, we will show in section~6 that
the gradient blow-up set is contained in a small neighborhood of the origin.

The following lemma shows in particular that gradient blow-up may occur
for initial data of arbitrarily small $L^\infty$-norm 
(but the $W^{1,\infty}$ norm has to be sufficiently large).
We note that we do not assume (\ref{donuy}) here, so that in the proof, we work only with weak solutions
(in particular, we cannot use the continuity of $\nabla u$ up to the boundary).

\begin{lem}\label{blowupdata}
Let $\kappa =(q-p)/(q-p+1)$. There exists $C_1=C_1(p,q)>0$ such that,
if $\eps\in (0,\min(L_1,L_2))$, $\mu\ge 0$ and $u_0\in \mathcal{V}_\mu$ satisfies
\begin{equation}\label{blowupdatacond}
u_0(x,y)\ge C_1\eps^\kappa \quad\text{in }B_{\eps/3}(0,\eps)\subset\Omega,
\end{equation}
then $T_{max}(u_0)<\infty$. 
\end{lem}

\begin{proof}
We denote $B_r=B_r(0,0)\subset\R^2$ for $r>0$.
Fix a radially symmetric function $h\in C^\infty_0(B_1)$, $h\ge 0$,
such that ${\rm supp}(h)\subset B_{1/3}$ and $\|h\|_\infty=1$.
We consider the following problem
\begin{equation}\label{pbpourbu}
\left\{
\begin{array}{lll}
v_t-\Delta_p v=|\nabla v|^q,&x\in B_1, &t>0\\
v(x,y,t)=0,& x\in \partial B_1, &t>0\\
v(x,y,0)=v_0(x,y):= C_1h(x,y), &x\in B_1.& 
\end{array}
\right.
\end{equation}
We know from \cite[Theorem~1.4]{AmalJDE} that there exists $C_0= C_0(p,q)>0$ such that if $\|v_0\|_{L^1}\geq C_0$, then $T_{max}(v_0)<\infty$. 
Therefore, if we take $C_1=C_1(p,q)>0$ large enough then $\nabla v$ blows up in finite time in $L^\infty$ norm.

Next we use the scale invariance of the equation, considering the rescaled functions 
$$v_{\eps}(x,y,t):=\eps^{\kappa} v\left(\dfrac{x}{\eps}, \dfrac{y-\eps}{\eps}, \dfrac{t}{\eps^{(2q-p)/(q-p+1)}}\right).$$
Pick $\eps\in (0,\min(L_1,L_2))$ and denote $\tilde B_{\eps}=B_{\eps}(0,\eps)$, 
which is included in $\Omega$ and tangent to $\partial \Omega$ at the origin. 
Set 
$T_{\eps}=\eps^{(2q-p)/(q-p+1)}T_{max}(v_0)$ and $\tilde{T}_\eps=\min \bigl(T_{max}(u_0), T_{\eps}\bigr)$.
We shall show that, for each $\tau\in (0, \tilde{T}_\eps)$,
\begin{equation}\label{goalNableveps}
 \|\nabla v_\eps\|_{L^\infty(Q_\tau)}
\le \max \Bigl(\|\nabla v_\eps(\cdot,0)\|_\infty, \|\nabla u\|_{L^\infty(Q_\tau)} \Bigr).
\end{equation}
Since gradient blow-up occurs in finite time $T_{\eps}$ for $v_\eps$, this will guarantee $T_{max}(u_0)\leq T_{\eps}<\infty$.

First observe that $v_{\eps}$ solves \eqref{pbpourbu} in $\tilde B_{\eps}\times (0, T_{\eps})$,
with initial data $v_{\eps}(x,y,0)\leq u_0(x,y)$, due to (\ref{blowupdatacond}).
It follows from the comparison principle that
$v_{\eps} \le u$ in $\tilde B_{\eps}\times(0, \tilde{T}_\eps)$.
In particular, for each $h\in (0,\eps)$ and $t\in (0, \tilde{T}_\eps)$, since $u(0,0,t)=0$,
we get that
\begin{equation}\label{hous}
 \dfrac{v_{\eps}(0,h,t)}{h} \le \dfrac{u(0,h,t)}{h}\le \| \nabla u(\cdot,t)\|_\infty.
\end{equation}
We shall next show that the 
first quantity in (\ref{hous}) can be suitably bounded from below in terms of the sup norm of $\nabla v_\eps$.

Fix $h\in (0,\eps)$ and let $\tilde B_{\eps}^h:=B_{\eps}(0,\eps-h)$
Since $v_{\eps}^h(x,y,t):=v_{\eps}(x, y+h,t)$ is a solution of \eqref{pbpourbu} in $\tilde B_{\eps}^h$,
it follows from the comparison principle (see Remark~\ref{remsou1}) that, for any $0<\tau<\tilde T_{\eps}$,
\begin{equation}\label{hous2}
\aligned 
&\displaystyle\sup_{(\tilde B_{\eps}\cap \tilde B_{\eps}^h)\times(0, \tau)}|v_{\eps}(x, y+h,t)-v_{\eps}(x, y, t)|
\qquad\qquad\\
&\leq \max\left(\underset{\tilde B_{\eps}\cap \tilde B_{\eps}^h}{\sup}\, |v_{\eps}^h(x,y, 0)-v_{\eps}(x, y,0)|,
\underset{\partial (\tilde B_{\eps}\cap \tilde B_{\eps}^h)\times(0, \tau)}{\sup}\, |v_{\eps}^h(x, y, t)-v_{\eps}(x,y,t)| \right).
\endaligned 
\end{equation}

We claim that, for any $0<t<T_\eps$,
\begin{equation}\label{hous3}
\underset{\partial (\tilde B_{\eps}\cap \tilde B_{\eps}^h)}{\sup}\, |v_{\eps}^h(x, y, t)-v_{\eps}(x,y,t)|
\le 
v_{\eps}(0,h,t).
\end{equation}
First consider the case $(x,y)\in\partial \tilde B_{\eps}$. Then $\|(x,y+h)-(0,\eps)\|\ge \|(0,h)-(0,\eps)\|$, due to
$$x^2+(y+h-\eps)^2-(h-\eps)^2
=\eps^2-(y-\eps)^2+(y+h-\eps)^2-(h-\eps)^2=2hy\ge 0.$$
Since $v_{\eps}$ is radially symmetric and non-increasing
with respect to the point $(0,\eps)$, we deduce that
$$|v_{\eps} (x,y+h,t)-v_{\eps}(x,y, t)|=v_{\eps}(x,y+h,t)\leq v_{\eps}(0,h, t).$$
Next consider the case $(x,y)\in\partial \tilde B^h_{\eps}$ that is, $(x, y+h)\in\partial \tilde B_{\eps}$. Then
$\|(x,y)-(0,\eps)\|\ge \|(0,h)-(0,\eps)\|$, due to
$$x^2+(y-\eps)^2-(h-\eps)^2=\eps^2-(y+h-\eps)^2+(y-\eps)^2-(h-\eps)^2
=2h(2\eps-h-y)\ge 0.$$
Therefore, $|v_{\eps} (x,y+h,t)-v_{\eps}(x,y, t)|=v_{\eps}(x,y, t)\leq v_{\eps}(0,h,t)$,
and the claim~(\ref{hous3}) is proved.

Now fix $y\in [0,\eps)$. It follows from (\ref{hous})-(\ref{hous3}) that, for each $0<h<\eps-y$ and $0<t<\tau<\tilde T_{\eps}$,
$$
\dfrac{|v_{\eps}(0,y+h,t)-v_{\eps}(0,y,t)|}{h}
\le \max \Bigl(\|\nabla v_\eps(\cdot,0)\|_\infty, \|\nabla u\|_{L^\infty(Q_\tau)} \Bigr)$$
hence, letting $h\to 0$,
$$
|\partial_y v_\eps(0,y,t)|
\le \max \Bigl(\|\nabla v_\eps(\cdot,0)\|_\infty,  \|\nabla u\|_{L^\infty(Q_\tau)} \Bigr).
$$
For each $0<\tau<\tilde T_{\eps}$, taking supremum over $y\in [0,\eps)$ and $t\in (0,\tau)$ and using the fact that $v_\eps$ is radially symmetric, we obtain
(\ref{goalNableveps}). This concludes the proof of the lemma.
\end{proof}

\section{Local boundary control for the gradient
and localization of the gradient blow-up set}

For simplicity we shall here assume \eqref{don0b2} and (\ref{donuy}), so as to have the continuity of $\nabla u$ up to the boundary
(although one might possibly relax this assumption at the expense of additional work).

\begin{lem}\label{localgrad}
Let $\rho, \mu, u_0$ be as in Lemma~\ref{uclass2}(ii) and
let $(x_0,y_0)\in\partial \Omega$.
If there exist $M_0, R>0$ such that 
\begin{equation}\label{hyplocalgrad}
|\nabla u|\leq M_0 \qquad\text{in }(B_R(x_0,y_0)\cap\partial\Omega)\times [0,T_{max}(u_0)),
\end{equation}
then $(x_0, y_0)$ is not a gradient blow-up point.
\end{lem}

The proof is based on a local Bernstein technique. 
For $(x_0,y_0)\in\partial\Omega$, $R>0$ 
and given $\alpha\in (0,1)$, we may select a cut-off function $\eta\in C^2(\overline{B}_{R}(x_0, y_0))$, with $0 < \eta \leq 1$, such that
$$\eta=1\ \hbox{ on $\overline{B}_{R/2}(x_0, y_0)$}, \qquad \eta = 0\ \hbox{ on $\partial B_{R}(x_0, y_0)$}$$
and
\begin{eqnarray}
\left.
\begin{array}{rr}
|\nabla \eta|\leq C R^{-1} \eta ^{\alpha}\\
\noalign{\vskip 1mm}
|D^2 \eta|+\eta^{-1}|\nabla \eta|^2\leq C R^{-2} \eta^{\alpha}
\end{array}
\right\}
\quad\text{on $\overline{B}_{R}(x_0, y_0)$},
\end{eqnarray}
where $C = C (\alpha) > 0$  (see~e.g.~\cite{Zhan} for an example of such function).
Also, for $0< t_0<\tau<T=T_{max}(u_0)$, we denote
$$Q^{t_0}_{\tau, R}= \left(B_R(x_0, y_0)\cap \Omega\right)\times (t_0,\tau).$$
For the proof of Lemma \ref{localgrad},
we then rely on the following lemma from \cite{AmalJDE}
(cf.~\cite[Lemma~3.1]{AmalJDE}; it was used there to derive upper estimates on $|\nabla u|$ away from the boundary).

\begin{lem}\label{tech}
Let $\mu\ge 0$ and $u_0\in \mathcal{V}_\mu$.
Let $(x_0,y_0)\in\partial\Omega$, $R>0$, $0< t_0<\tau<T$ and
choose $\alpha=(q+1)/(2q-p+2)$.
Denote $w= |\nabla u|^2$ and $z=\eta w$.
Then $z\in C^{2,1}(Q_{\tau,R}^{t_0})$ and satisfies the following differential inequality
\begin{eqnarray}\label{ouf}
\mathcal{L} z+ C_2 z^{\frac{2q-p+2}{2}}\leq C_3\left(\dfrac{\left\|u_0\right\|_{\infty}}{t_0}\right)^{\frac{2q-p+2}{q}} + C_3R^{-\frac{2q-p+2}{q-p+1}},
\end{eqnarray}
where $C_i=C_i(p, q)>0$,
\begin{eqnarray}
&\mathcal{L} z=z_t- \bar{\mathcal{A}} z-\bar{H} \cdot\nabla z,\\
\noalign{\vskip 1mm}
&\bar{\mathcal{A}} z= |\nabla u|^{p-2} \Delta z + (p-2) |\nabla u|^{p-4}(\nabla u)^t D^2 z \nabla u,
\end{eqnarray}
and $\bar{H}$ is defined by 
\begin{align}
& \bar{H}:= \left[ (p-2) w^{\frac{p-4}{2}}\Delta u+\dfrac{(p-2)(p-4)}{2} w^{\frac{p-6}{2}} \nabla u \cdot\nabla w+q w^{\frac{q-2}{2}}\right]\nabla u \nonumber \\
 &\quad\quad+\dfrac{p-2}{2} w^{\frac{p-4}{2}} \nabla w. \label{apender}
\end{align}
 \end{lem}

\begin{proof}[{Proof of Lemma \ref{localgrad}}]
Let $t_0=T/2<\tau<T$ and set
$$M_1:=\underset{0\leq t\leq t_0}{\text{sup}} \left\|\nabla u\right\|_{L^{\infty}}<\infty.$$
By Lemma~\ref{uclass2}(ii) and Theorem \ref{theorsou2},
we know that $\nabla u$ is a continuous function on $\overline{\Omega}\times (0,T)$,
hence $z\in C(\overline{Q_{\tau,R}^{t_0}})$.
Therefore, unless $z\equiv 0$ in $\overline{Q_{\tau, R}^{t_0}}$,  $z$  must reach a positive maximum at some point $(x_1,y_1,t_1)\in \overline{Q_{\tau,R}^{t_0}}$.  
Since 
\begin{equation}\label{boundz1}
z=0\quad\hbox{ on $\left(\partial B_{R}(x_0,y_0)\cap\overline\Omega\right)\times [t_0, \tau]$,}
\end{equation}
 we deduce that either 
$(x_1,y_1)\in B_{R}(x_0,y_0)\cap\Omega$ or $(x_1, y_1)\in B_R(x_0,y_0)\cap\partial\Omega$.
\smallskip

$\bullet$ If $t_1=t_0$, then 
\begin{equation}\label{boundz2}
z(x_1, y_1, t_1)\leq \left\|\nabla u(t_0)\right\|_{L^{\infty}}^2\leq M_1^2.
\end{equation}

$\bullet$ If $t_0< t_1\leq \tau$ and $(x_1, y_1)\in B_R(x_0,y_0)\cap\partial\Omega$, then, by (\ref{hyplocalgrad}),
\begin{equation}\label{boundz3}
z(x_1, y_1, t_1)\leq M_0^2.
\end{equation}

$\bullet$ Next consider the case $t_0< t_1\leq \tau$ and $(x_1,y_1)\in B_R(x_0,y_0)\cap\Omega$.
Then we have $\nabla z(x_1, y_1,  t_1)=0$, $z_t(x_1, y_1, t_1)\geq 0$ and $D^2 z(x_1, y_1, t_1)\leq 0$,
and therefore $\mathcal{L} z\geq0$.
Using~\eqref{ouf} we arrive at
$$
C_2  z(x_1, y_1, t_1)^{\frac{2q-p+2}{2}}\leq C_3 \left(\dfrac{\left\|u_0\right\|_{\infty}}{t_0}\right)^{\frac{2q-p+2}{q}}+ C_3 R^{-\frac{2q-p+2}{q-p+1}}
$$
that is,
\begin{equation}\label{boundz4}
\sqrt{z(x_1, y_1, t_1)}\leq C\left(\dfrac{\left\|u_0\right\|_{\infty}}{t_0}\right)^{\frac{1}{q}} + CR^{-\frac{1}{q-p+1}}=:M_2>0.
\end{equation}

It follows from (\ref{boundz1})-(\ref{boundz4}) that 
$$\underset{\overline{Q_{\tau,R}^{t_0}}}{\max}\ z\leq M_3^2, \quad\hbox{ with $M_3=\max\left\{M_0, M_1, M_2\right\}$}.$$
Since $z=|\nabla u|^2$ in $\left(B_{R/2}(x_0,y_0)\cap\overline\Omega\right)\times(t_0,\tau)$ and $\tau\in(t_0,T)$ is arbitrary, we get
$$|\nabla u|\leq M_3 \quad\hbox{in $\left(B_{R/2}(x_0,y_0)\cap\overline\Omega\right)\times(t_0,T)$,}$$
and we conclude that $(x_0,y_0)$ is not a gradient blow-up point.
\end{proof}

By combining Lemmas \ref{uclass2} and \ref{localgrad}, we can now easily obtain a class of initial data 
whose possible gradient blow-up set is contained in a small neighborhood of the origin.

\begin{lem}\label{localiz}
Let $\rho, \mu, u_0$ be as in Lemma~\ref{uclass2}(ii).
Then $GBUS(u_0) \subset [-\rho, \rho]\times\left\{0\right\}$.\end{lem}

\begin{proof} Denote again $\Sigma_\rho=[-\rho/2,\rho/2]\times \{0\}$ and $\Sigma'_\rho=\partial\Omega\setminus ([-\rho,\rho]\times \{0\})$.
In view of  Lemma \ref{localgrad}, it suffices to show that
\begin{equation}\label{BoundSigmaRho}
\sup_{(x,y)\in\Sigma'_\rho,\  t\in (0,T)} |\nabla u(x,y,t)| <\infty.
\end{equation}

But (\ref{BoundSigmaRho}) easily follows from a comparison with the function $\bar U$
provided in Lemma~\ref{GlobalBarrier}.
Indeed, under the assumptions of Lemma \ref{uclass2}(i), we already know that $u\le \bar U$ in $Q_T$.
Also, $u=\mu y= \bar U$ on $\Sigma'_\rho\times (0,T)$.
From this, along with (\ref{compuh0}), it follows that 
\begin{equation}\label{BoundSigmaRho2}
 \partial_\nu\bar U\le \partial_\nu u \le \mu \,\partial_\nu y\ \text{ on $\Sigma'_\rho\times (0,T)$.}
\end{equation}
From (\ref{BoundSigmaRho2}) and (\ref{eqprincipale})$_2$, we get
$$ |\nabla u|^2\le \mu^2+|\partial_\nu u|^2 \le C \ \text{ on $\Sigma'_\rho\times (0,T)$,}$$
hence (\ref{BoundSigmaRho}), and the lemma is proved.
\end{proof}

\section{Existence of well-prepared initial data: proof of Theorem \ref{singlegbu}(i)}

We need to construct initial data
meeting the requirements from sections~3--5.
This will be achieved in the following lemma.
Let us fix an even function $\varphi\in C^{\infty}(\R)$ such that $s\varphi'(s)\le 0$, with
\begin{equation}\label{defvarphicutoff}
\varphi(s)=\left\{\begin{array}{ll}
       1\quad &\text{for\ } |s|\le 1/3\\
              \noalign{\vskip 1mm}
	0 \quad&\text{for\ } |s|\ge 2/3.
      \end{array}
\right.
\end{equation}

\begin{lem}\label{bonnedonnee}
Let $\kappa=(q-p)/(q-p+1)$ and let $C_1=C_1(p,q)>0$ be given by Lemma~\ref{blowupdata}.
 For $\eps\in (0,\min(L_1,L_2/2))$, define
\begin{equation}\label{defpsicutoff}
\psi_\eps(y)=\left\{\begin{array}{ll}
       \varphi\Bigl(\displaystyle\frac{y-\eps}{\eps}\Bigr)\quad &\text{for} \ 0\le y\le \eps\\
       \noalign{\vskip 1mm}
	\varphi\Bigl(\displaystyle\frac{y-\eps}{L_2}\Bigr) \quad&\text{for} \ y\ge \eps\\
      \end{array}
\right.
\end{equation}
and let $u_0$ be defined by
$$u_0(x,y)=\mu y+C_1\eps^\kappa\varphi\Bigl(\displaystyle\frac{x}{\eps}\Bigr)\psi_\eps(y).$$
Next fix $0<\rho<L_1$ and let $\mu_0=\mu_0(p,q,\Omega,\rho)>0$ and $c=c(p,q,\Omega,\rho)>0$ 
be given by Lemmas~\ref{GlobalBarrier} and \ref{uclass2}.
For any $\mu\in (0,\mu_0]$, there exists $\eps_0=\eps_0(p,q,\Omega,\mu,\rho)>0$
such that, for all $\eps\in(0,\eps_0]$, the function $u_0\in \mathcal{V}_\mu$ and satisfies 
\begin{eqnarray}
&&\hbox{$u_0$ is symmetric with respect to the line $x=0$}, \label{don00} \\
&&\partial_x u_0\leq 0 \ \text{ in $\Omega_+$}, \label{don0b} \\
&&\partial_y u_0\geq \mu/2 \quad\text{in }\Omega, \label{don1b} \\
&&u_0(x,y)\le\mu \bigl(y+c \chi_{(-\rho/2,\rho/2)\times (0,L_2)}\bigr) 
\quad\text{in }\Omega, \label{don0b2c}\\
&&u_0(x,y)\ge C_1 \eps^\kappa \quad\text{in }B_{\eps/3}(0,\eps)\subset\Omega. \label{don4b}
\end{eqnarray}
\end{lem}

\begin{proof}
Assume $\eps\le \min(L_1,L_2/12)$. Then 
\begin{equation}\label{controlPsiEps1}
\psi_\eps(y)=0 \quad\text{for\ } y\not\in [\textstyle\frac{\eps}{3},\frac{3L_2}{4}]
\end{equation}
(indeed, $y\ge \frac{3L_2}{4}$ implies $\frac{y-\eps}{L_2}\ge \frac{3}{4}-\frac{1}{12}=\frac{2}{3}$)
and therefore $u_0\in \mathcal{V}_\mu$.
Properties (\ref{don00})-(\ref{don0b}) are clear by the choice of $\varphi$.

To check (\ref{don1b}), we note that 
$$\partial_y u_0 =\mu +C_1\varepsilon^\kappa\varphi\Bigl(\displaystyle\frac{x}{\eps}\Bigr)\psi'_\eps(y).$$
For  $0\leq y\leq \eps$, we have $\psi'_\eps(y)\ge 0$, hence $\partial_y u_0 \ge\mu$.
Whereas, for $y\ge \eps$, we have 
$$\psi'_\eps(y)=L_2^{-1}\varphi'\bigl((y-\eps)/L_2\bigr)\ge -L_2^{-1}\|\varphi'\|_\infty,$$
hence 
$$\partial_y u_0 \ge \mu-C_1\varepsilon^\kappa L_2^{-1}\|\varphi'\|_\infty\ge \mu/2$$
whenever $\varepsilon^\kappa\le \mu L_2/(2C_1\|\varphi'\|_\infty)$.

As for (\ref{don0b2c}), if $C_1\varepsilon^\kappa\le \mu c$ and $\varepsilon\le\rho/2$,
it immediately follows from $\varphi,\psi_\eps\le 1$ and ${\rm supp}(\varphi)\subset (-1,1)$.
Finally, since $\varphi(x/\eps)=1$ for $|x|\le \eps/3$ and $\psi_\eps(y)=1$ for $|y-\eps|\le \eps/3$,
we have (\ref{don4b}). The lemma is proved.
\end{proof}

\begin{proof}[Proof of Theorem \ref{singlegbu}(i)]
Let $\mu$ and $u_0$ be as in Lemma~\ref{bonnedonnee}.

$\bullet$ The fact that $T_{max}(u_0) < \infty$ follows from Lemma~\ref{blowupdata}.

$\bullet$ Next, we have $GBUS(u_0)\subset [-\rho, \rho]\times\left\{0\right\}$
as a consequence of Lemma~\ref{localiz}.

$\bullet$ Properties (\ref{don3})-(\ref{don4}) 
follow from Lemma \ref{uclass}.

$\bullet$ Finally, property (\ref{don1}) is a consequence of Lemma \ref{uclass2}(ii).

This proves the assertion.
\end{proof}

\section{Nondegeneracy of gradient blow-up points}

In this section, we show that if $u$ is only ``weakly singular'' in a neighborhood of a boundary point
$(x_0, 0)$, then the singularity is removable.

\begin{lem}\label{milouar1}
Let $\rho, \mu, u_0$ be as in Lemma~\ref{uclass2}(ii) and
let $x_0\in (-L_1,L_1)$. There exist $c_0=c_0(p,q)>0$ such that,
if $u_0\in \mathcal{V}_\mu$ with $T:=T_{max}(u_0)<\infty$ and 
\begin{equation}  
u(x,y)\leq  c_0 y ^{(q-p)/(q-p+1)}\qquad\text{in}\quad \left(B_R(x_0,0)\cap \Omega\right)\times [t_0, T), 
\end{equation} 
for some $R>0$ and $t_0\in (0,T)$, then $(x_0, 0)$ is not a gradient blow-up point.
\end{lem}

\begin{proof}
Let $x_0\in (-L_1,L_1)$. Then for some constants $r\in (0, R)$ and $d\in (0,L_2)$,
we  have  that $$\omega_1:=\left\{(x,y)\in \mathbb{R}^2; |x-x_0|<r,\ 0<y<d\right\}\subset B_{R}(x_0,0)\cap\,\Omega.$$
Setting $\beta=1/(q-p+1)$, we define the comparison function
$$v=v(x,y,t)=\eps y V^{-\beta}\qquad\text{in }Q:=  \overline{\omega_1}\times(t_0, T)$$
with
$$V=y+\eta \left( r^2-(x-x_0)^2\right)(t-t_0),$$
where $\eta, \eps>0$ are to be determined later.
We compute, in $Q$,
$$v_t=- \eps\beta\eta y (r^2-(x-x_0)^2)V^{-\beta-1},$$
$$v_x=2 \eps\beta\eta y  (x-x_0) (t-t_0)V^{-\beta-1},$$
$$v_y=\eps V^{-\beta}- \eps\beta y V^{-\beta-1}
=\eps V^{-\beta} \Bigl[1-\beta \frac{y}{V}\Bigr],
$$
\begin{eqnarray*}
v_{xx}
&=&2 \eps\beta\eta y (t-t_0)V^{-\beta-1}-4 \eps\beta\eta^2 y  (x-x_0)^2 (t-t_0)^2(-\beta-1)V^{-\beta-2} \\
&=&2 \eps\beta\eta (t-t_0)V^{-\beta-1}\Bigl[y+2(\beta+1)\eta (x-x_0)^2(t-t_0)\frac{y}{V}\Bigr],
\end{eqnarray*}
$$
v_{yy}
=-2\eps\beta V^{-\beta-1}+\eps\beta(\beta+1) y V^{-\beta-2}
=\eps\beta V^{-\beta-1} \Bigl[-2+(\beta+1) \frac{y}{V}\Bigr],
$$
\begin{eqnarray*}
v_{xy}
&=&2 \eps\beta\eta  (x-x_0) (t-t_0)V^{-\beta-1}-2 \eps\beta(\beta+1)\eta y  (x-x_0) (t-t_0)V^{-\beta-2} \\
&=&2 \eps\beta\eta  (x-x_0) (t-t_0)V^{-\beta-1} \Bigl[1-(\beta+1)\frac{y}{V}\Bigr].
\end{eqnarray*}
Noting that $\beta <1$ and  $\dfrac{y}{V}\leq 1$, we see that, in $Q$,
$$0\le v_{xx}\le 2 \eps\beta\eta T\bigl(d+2(\beta+1)\eta r^2T\bigr)V^{-\beta-1},
\qquad
|v_{xy}| \le 2 \eps\beta\eta  rT V^{-\beta-1}
$$
and
$$
v_{yy}\le \eps \beta(\beta-1)V^{-\beta-1}<0.$$
It follows that
\begin{eqnarray*}
\Delta_p v
&=&|\nabla v|^{p-2}\Bigl[\Delta v+(p-2)\dfrac{v_iv_jv_{ij}}{|\nabla v|^2}\Bigr]
\le |\nabla v|^{p-2}\bigl[(p-1)v_{xx}+v_{yy}+(p-2)|v_{xy}|\bigr] \\
&\le &\eps \beta |\nabla v|^{p-2}V^{-\beta-1}\Bigl[2 (p-1)\eta T\bigl(d+2(\beta+1)\eta r^2T\bigr)+(\beta-1)+2(p-2)\eta rT\Bigr].
\end{eqnarray*}
On the other hand, we have
$$|\nabla v|\ge |v_y| \ge \eps(1-\beta) V^{-\beta} \ge \eps(1-\beta)(d+\eta T r^2)^{-\beta},$$
hence
$$v_t \ge - \eps\beta\eta dr^2V^{-\beta-1}
\ge -\eps\beta  |\nabla v|^{p-2}V^{-\beta-1}\Big[\eta d r^2\left((1-\beta)\eps\right)^{2-p}(d+\eta T r^2)^{(p-2)\beta}\Bigr].$$
Therefore, 
\begin{eqnarray*}
&&v_t-\Delta_p v 
\ge \eps\beta  |\nabla v|^{p-2}V^{-\beta-1}\times \Big[-\eta d r^2\left((1-\beta)\eps\right)^{2-p}(d+\eta T r^2)^{(p-2)\beta} \\
&&\qquad\qquad\qquad\qquad\qquad-2 (p-1)\eta T\bigl(d+2(\beta+1)\eta r^2T\bigr)-2(p-2)\eta rT+(1-\beta)\Bigr].
\end{eqnarray*}
Since also
$$
|v_x| \le 2\eps\beta\eta rTV^{-\beta},\quad
|v_y|  \le\eps V^{-\beta}, $$
if we choose $\eta=\eta(p,q,d,r,T, \eps)>0$ small enough, we get that, in $Q$,
$$|\nabla v|\le 2\eps V^{-\beta}$$
and
$$
v_t-\Delta_p v 
\ge \dfrac{\eps\beta(1-\beta)}{2} |\nabla v|^{p-2}V^{-\beta-1},
$$
hence
$$
v_t-\Delta_p v \ge \dfrac{\eps\beta(1-\beta)}{2} |\nabla v|^{p-2}(2\eps)^{-\frac{\beta+1}{\beta}}|\nabla v|^{\frac{\beta+1}{\beta}}
= \dfrac{\beta(1-\beta)}{4} (2\eps)^{-\frac{1}{\beta}}|\nabla v|^q,$$
due to $\beta=1/(q-p+2)$.
If $\eps=\eps_0(p,q)>0$ is small enough, we thus obtain 
\begin{equation}
v_t-\Delta_p v\geq |\nabla v|^q.
\end{equation}

\smallskip

Now we shall check the comparison on  the parabolic boundary of $\omega_1\times (t_0,T)$.
On $\omega_1\times\left\{t_0\right\}$, choosing $c_0=2^{-\beta}\eps_0$, we have
\begin{equation}\label{initial}
u\leq c_0 y^{1-\beta}= 2^{-\beta}\eps_0 y^{1-\beta}\leq v.
\end{equation}  
On the lateral boundary  part $\left\{ (x,y)\in \mathbb{R}^2; |x-x_0|=r, 0\leq y\leq d\right\}\times (t_0,T)$, inequality \eqref{initial} holds also.  On the surface $\left\{ (x,y)\in \mathbb{R}^2; |x-x_0|\leq r,  y=0\right\}\subset \partial\Omega$, we have for $t_0<t<T$,
$$u(.,.,t)=v(.,., t)=0$$
Finally, on $\left\{ (x,y)\in \mathbb{R}^2; |x-x_0|\leq r, y=d\right\}\times (t_0,T)$, assuming in addition that 
$\eta$ satisfies $\eta\leq dT^{-1}r^{-2}$, we get 
$$u\leq c_0 d^{1-\beta}\leq \eps_0d(d+\eta r^2T)^{-\beta}\leq v.$$
Using the comparison principle, we get that 
\begin{equation}
u\leq v \qquad\text{in }\omega_1\times(t_0, T).
\end{equation}
This implies that 
$$|u_y|\leq \eps \left(\eta(r^2-|x-x_0|^2)(t-t_0)\right)^{-\beta}\leq  M_0$$
on $(B_{r/2}((x_0,0))\cap\partial\Omega)\times((t_0+T)/2,T)$
 for some constant $M_0>0$. Lemma \ref{milouar1} is then a direct consequence of Lemma \ref{localgrad}.
 \end{proof}

\section{The auxiliary function $J$ and the proof of single-point gradient blow-up}

In all this section, we fix $\rho, x_1$ with
\begin{equation}\label{defrhox1}
0<\rho< x_1< L_1
\end{equation}
and we assume that $\mu$ and $u_0\in\mathcal{V}_\mu\cap C^2(\overline\Omega)$
satisfy the assumption of Theorem~\ref{singlegbu}(ii)
i.e., the corresponding solution of (\ref{eqprincipale}) fulfills properties (\ref{don0})-(\ref{don1}).
We denote as before $T=T_{max}(u_0)$.

We consider the auxiliary function 
	$$ 
	J(x,y,t):=u_x+ c(x)d(y) F(u), 
	$$ 
with
	$$ 
	\left\{
	\begin{array}{lll}
	F(u)&=& u^{\alpha},\\
	c(x)&=&kx,\qquad k>0,\\
	d(y)&=&y^{-\gamma}
	\end{array}
	\right.
	$$ 
and
	\begin{equation}\label{defFalphagamma}
	1<\alpha<1+q-p,\qquad \gamma=(1-2 \sigma)(\alpha -1),
	\end{equation}
where	
\begin{equation}\label{defsigma}
	0<\sigma < \dfrac{1}{2(q-p +1)}
	\end{equation}
is fixed.
Letting
$$D=(0,x_1)\times(0,y_1),$$
our goal is to use a comparison principle to prove that
$$	J\leq 0 \quad\text{in } D\times (T/2, T),$$
provided $\alpha>1$ is chosen close enough to $1$ (hence making $\gamma>0$ small)
and $y_1\in (0,L_2)$ and $k>0$ are chosen sufficiently small.

 \subsection{Parabolic inequality for the auxiliary function $J$}
 
By the regularity of $u$ 
(see Theorem \ref{theorsou2}), we have
$$J\in C^{2,1}(Q_T).$$
A key step is to derive a parabolic inequality for $J$.
To this end, we define the operator
\begin{equation}\label{DefOperP}
\mathcal {P} J:=J_t- |\nabla u|^{p-2} \Delta J-(p-2)|\nabla u|^{p-4}\left\langle D^2J\,\nabla u,\nabla u\right\rangle
+\mathcal{H }\cdot\nabla J+\mathcal{A}J,
\end{equation}
where the functions $\mathcal{H}=\mathcal{H}(x,y,t)$ and $\mathcal{A}=\mathcal{A}(x,y,t)$
are given by formulae (\ref{defH})--(\ref{defA1b}) below.

\begin{prop}\label{PJneg}
Assume (\ref{defrhox1}), (\ref{defsigma}) and let $\mu, u_0$ satisfy the assumption of Theorem~\ref{singlegbu}(ii).
There exist $\alpha,\gamma$ satisfying (\ref{defFalphagamma}), $y_1\in (0,L_2)$ and $k_0>0$,
all depending only on $p,q,\Omega,\mu,\|u_0\|_{C^2},\sigma$, such that, for any $k\in (0,k_0]$, 
the function $J$ satisfies
\begin{equation}
\mathcal{P}J\leq 0\qquad\text{in }   D\times (T/2, T).
\end{equation}
Moreover, 
 \begin{equation}
 \label{propAH}
 \mathcal{H}, \mathcal{A}\in C(D\times (0, T))\quad\hbox{and}\quad
\mathcal{A}\in L^\infty(D\times (T/2, \tau)) \quad\hbox{for each $\tau\in (T/2,T)$.}
 \end{equation}
\end{prop}

The proof of Proposition~\ref{PJneg} is very long and technical.
In order not to disrupt the main line of argument,
we postpone it to section~9 and now present the rest of the proof of Theorem~\ref{singlegbu}(ii).

\subsection{Boundary conditions for the auxiliary function $J$}

The verification of the appropriate boundary and initial conditions for the function $J$
depends on an essential way on the applicability 
to $u_x$ of the Hopf boundary lemma
at the points $(x_1,0)$ and $(0,y_1)$, up to $t=T$.
To this end, besides the nondegeneracy of problem (\ref{eqprincipale}), guaranteed by (\ref{don1}),
we also need the following local regularity lemma, which ensures that $D^2u$ remains bounded up to $t=T$
away from the gradient blow-up set.

\begin{lem}\label{lemRegulBoundaryCondJ}
Let $\rho\in (0,L_1)$ and let $\mu, u_0$ satisfy the assumption of Theorem~\ref{singlegbu}(ii).
Let $\omega'\subset\omega\subset\Omega$
be such that $dist(\omega',\Omega\setminus\omega)>0$ and $t_0\in (0,T)$.
If 
$$\sup_{\omega\times(0,T)}|\nabla u|<\infty,$$
then
$$\sup_{\omega'\times(t_0,T)}|D^2u|<\infty.$$
\end{lem}

\begin{proof}
Introduce an intermediate domain $\omega''$ with 
$\omega'\subset\omega''\subset\omega$, such that $dist(\omega'',\Omega\setminus\omega)>0$
and $dist(\omega',\Omega\setminus\omega'')>0$.
Write the PDE in (\ref{eqprincipale})  as
$$-\nabla\cdot(|\nabla|^{p-2}\nabla u)=|\nabla u|^q-u_t.$$
Using $|\nabla u|\ge\partial_yu\ge \mu/2$ (cf.~(\ref{don1})) and the boundedness of $u_t$ in $Q_T$ (cf. Lemma~\ref{estimderivtemps}),
it follows from the elliptic estimate in \cite[Theorem~V.5.2]{lady} that there exists $\theta\in (0,1)$ such that
$$\|\nabla u(\cdot,t)\|_{C^\theta(\overline\omega'')}\le C,\quad t_0/2\le t<T.$$
The boundedness of $u_t$ in $Q_T$ and the interpolation result in \cite[Lemma~II.3.1]{lady} then guarantee the estimate
$$\|\nabla u(x,y,\cdot)\|_{C^\beta([t_0/2,T))}\le C,\quad (x,y)\in \overline\omega',$$
where $\beta=\theta/(1+\theta)$.
Therefore $\|\nabla u\|_{C^\beta(\overline\omega'\times [t_0/2,T))}\le C$.
The conclusion now follows by applying standard Schauder parabolic estimates to the PDE in (\ref{eqprincipale}),
rewritten under the form
$$ 
u_t-a_{ij}u_{ij}=f,
\quad\hbox{where }
a_{ij}=|\nabla u|^{p-2}\Bigl[\delta_{ij}+(p-2)\dfrac{u_iu_j}{|\nabla u|^2}\Bigr],\quad f=|\nabla u|^q.
$$ 
Indeed, the matrix $(a_{ij})=(a_{ij}(x,y,t))$ is uniformly elliptic due to (\ref{don1}) and
\begin{equation}\label{UnifEllipt}
\begin{array}{lll}
a_{ij}\xi_i\xi_j
&=&|\nabla u|^{p-2}|\xi|^2+(p-2)|\nabla u|^{p-4}|\nabla u\cdot\xi|^2\\
\noalign{\vskip 2mm}
&\ge& |\nabla u|^{p-2}|\xi|^2\geq (\mu/2)^{p-2}|\xi|^2,
\end{array}
\end{equation}
and there exists $\nu\in(0,1)$ such that $a_{ij},f\in C^\nu(\overline\omega'\times [t_0/2,T'])$,
for each $T'<T$, with norm independent of $T'$.
\end{proof}

\begin{lem}\label{lemBoundaryCondJ}
Assume (\ref{defrhox1})--(\ref{defsigma}), let $y_1\in (0,L_2)$ and let $\mu, u_0$ satisfy the assumption of Theorem~\ref{singlegbu}(ii).
Then
\begin{equation}\label{contJ}
J\in C(\overline{D}\times(0,T))
\end{equation}
and there exists $k_1>0$ (depending in particular on $y_1$) 
such that, for any $k\in (0,k_1]$, the function $J$ satisfies
\begin{equation}\label{boundaryJ}
J\le 0\quad\text{on $\partial D\times (T/2,T)$}.
\end{equation}
\end{lem}

\begin{proof}
Since $u=0$ for $y=0$ and $|\nabla u|\leq C(\tau)$ in $\Omega\times [0,\tau]$ for each $\tau<T$, we have
$$u\leq C(\tau)y \qquad\text{in } D\times [0, \tau].$$
Due to $\gamma<\alpha$, we may therefore extend the function $c(x)d(y) F(u)$ continuously to be $0$ for $y=0$.
Property (\ref{contJ}) then follows from the regularity of $u$ 
 (see Theorem \ref{theorsou2})
and we have
\begin{equation}\label{atsou1}
J=0  \quad\text{on } (0,x_1)\times\left\{0\right\}\times(T/2,T).
\end{equation}

By (\ref{don4}), we have
$$u_x=0\quad\text{on  }  \{0\}\times (0,y_1)\times(0,T),$$
hence
\begin{equation}\label{atsou2}
J=0  \quad\text{on } \left\{0\right\}\times(0, y_1)\times(T/2,T).
\end{equation}

Next, the function $w=u_x$ is $\le 0$ in $\Omega^+\times(0, T)$ (cf.~(\ref{don4})) and satisfies there:
\begin{equation}\label{eqnw}
w_t= a_{ij}(x,y,t) w_{ij}+ B(x,y,t)\cdot\nabla w,  
\end{equation}
with
\begin{eqnarray*}
a_{ij}(x,y,t)&=&|\nabla u|^{p-2}\left[\delta_{ij}+(p-2)\dfrac{u_iu_j}{|\nabla u|^2}\right],\\
B(x,y,t)&=&(p-2)|\nabla u|^{p-4}\nabla u\Delta u+(p-2)(p-4)|\nabla u|^{p-6}\left\langle D^2 u\,\nabla u, \nabla u\right\rangle \nabla u\\
&+& q|\nabla u|^{q-2}\nabla u +2(p-2)|\nabla u|^{p-4}( D^2 u\,\nabla u).
\end{eqnarray*}
Fix $\rho<x_3<x_2<x_1$.
Since $GBUS\subset [-\rho, \rho]\times\left\{0\right\}$, we have 
$$|\nabla u|\leq C\quad\hbox{ in }\bigl(\Omega\backslash\left\{(-x_3, x_3)\times(0, y_1/3)\right\}\bigr)\times(0,T).$$
It follows from Lemma~\ref{lemRegulBoundaryCondJ} that
$$|D^2 u|\leq C\quad\hbox{ in }\bigl(\Omega\backslash\left\{(-x_2, x_2)\times(0, y_1/2)\right\}\bigr)\times(T/4,T),$$
hence 
$$|B|\leq C\quad\hbox{ in }\bigl(\Omega\backslash\left\{(-x_2, x_2)\times(0, y_1/2)\right\}\bigr)\times(T/4,T).$$
Moreover, the matrix $A(x,y,t)$ is uniformly elliptic (cf.~(\ref{UnifEllipt})).
We may thus apply the strong maximum principle and the Hopf boundary point Lemma \cite[Theorem~6~p.~174]{PW},
to get
\begin{eqnarray*}
u_x&\leq& -c_1y\quad\text{on } \left\{x_1\right\}\times(0,y_1)\times(T/2,T),\\
u_x&\leq&-c_1x \quad\text{on } (0,x_1)\times\left\{y_1\right\}\times(T/2,T).
\end{eqnarray*}
Also, since $x_1>\rho$ and $u(x,0,t)=0$, we get that, for some $c_2>0$,
$$
u\leq c_2y\quad\text{on } \{x_1\}\times(0,y_1)\times(T/2,T).
$$
Consequently, using $\alpha>\gamma+1$ and (\ref{minmaxu_0}), we have for $0<k\le k_1(y_1)$ sufficiently small
\begin{eqnarray*}
J(x, y_1,t)&\leq&-c_1x+kxy_1^{-\gamma} \left\|u_0\right\|_{\infty}^{\alpha}\le 0 \quad\text{on } (0,x_1)\times\left\{y_1\right\}\times(T/2,T),\label{atsou3}\\
J(x_1, y,t)&\leq&-c_1y+kx_1y^{\alpha-\gamma}c_2^{\alpha}\le 0 \quad\text{on } \{x_1\}\times(0,y_1)\times(T/2,T).\label{atsou4}
\end{eqnarray*}
This, along with (\ref{atsou1})-(\ref{atsou2}), proves (\ref{boundaryJ}).
\end{proof}

\subsection{Initial conditions for $J$}

\begin{lem}\label{lemInitCondJ}
Assume (\ref{defrhox1})--(\ref{defsigma}) and let $\mu, u_0$ satisfy the assumption of Theorem~\ref{singlegbu}(ii).
There exists $k_2>0$ such that, for any $k\in (0,k_2]$, the function $J$ satisfies
$$J(x,y,T/2)\le 0\quad\text{in $[0,x_1]\times[0,L_2]$}.$$
\end{lem}

The proof relies on a parabolic version of the Serrin corner lemma applied to $u_x$.
This is provided by Proposition~\ref{CornerLemma},
which we state and prove in Appendix 2.

\begin{proof}
The function $z=u_x$ satisfies equation (\ref{eqnw}).
We shall apply Proposition~\ref{CornerLemma} to this equation,
with $\tau_1=T/4$, $\tau_2=3T/4$,  $X_1=x_1$, $Y_1=L_2$, $\hat X_1=L_1$, $\hat Y_1=2L_2$.
We thus need to check the assumption \eqref{mikef}. 
Let us denote $\hat D_T=(0, \hat {\cb X_1})\times (0, \hat {\cb Y}_1)\times (T/4,3T/4)$.

For $x=0$ or $y=0$, we have $u_x=0$, hence $a_{12}=a_{21}=(p-2)|\nabla u|^{p-2}u_xu_y=0$. 
Due to the regularity of $u$ (cf.~(\ref{uC21})), we deduce that
\begin{equation}\label{hypcorner1}
a_{12}+a_{21}\ge -C(x\wedge y)\quad\hbox{ in $\overline{\hat D_T}$}.
\end{equation}

On the other hand, for $x=0$ and $0<y<\hat Y_1$, we have $u_{xy}=u_x=0$. 
Also, by (\ref{uxC21}), we have $|(u_{xy})_x|=|(u_x)_{yx}|\le C$ in~$\overline{\hat D_T}$.
Consequently
$|u_x|+|u_{xy}|\le Cx$ in~$\overline{\hat D_T}$. Using (\ref{uC21}) and (\ref{don1}), we deduce
\begin{align}
B_1
&=(p-2)|\nabla u|^{p-4}(\Delta u)u_x+(p-2)(p-4)|\nabla u|^{p-6}\left\langle D^2u\, \nabla u, \nabla u\right\rangle u_x \notag \\
&\qquad +q|\nabla u|^{q-2}u_x +2(p-2)|\nabla u|^{p-4}\bigl(u_{xx}u_x+u_{xy}u_y\bigr) \notag \\
&\ge -Cx
\quad\hbox{ in $\overline{\hat D_T}$}. \label{hypcorner2}
\end{align}

Next, for $y=0$ and $0<x<\hat X_1$, we have $u_t=0$ and $u_x=u_{xx}=0$. Recalling (\ref{uC21}), we thus have
\begin{align*}
(u_y)^q&=|\nabla u|^q=u_t-\Delta_p u=-|\nabla u|^{p-2}\left[\Delta u+(p-2)\dfrac{\left\langle D^2u\, \nabla u, \nabla u\right\rangle}{|\nabla u|^2}\right]\\
&=-|\nabla u|^{p-2}\left[u_{xx}+u_{yy}
+(p-2)\dfrac{u_{xx}u_x^2+2u_{xy}u_xu_y+u_{yy}u_y^2}{u_x^2+u_y^2}\right]
=-(p-1)(u_y)^{p-2}u_{yy}.
\end{align*}
It follows that, for $0<x<\hat X_1$ and $t\in [T/4,3T/4]$,
\begin{align*}
B_2(x,0,t)
&=(p-2)(p-4)|\nabla u|^{p-6}\bigl(u_{xx}u_x^2+2u_{xy}u_xu_y+u_{yy}u_y^2\bigr)u_y\\
&\quad +(p-2)|\nabla u|^{p-4}(u_{xx}+u_{yy})u_y+q|\nabla u|^{q-2}u_y \\
&\quad +2(p-2)|\nabla u|^{p-4}\bigl(u_{yx}u_x+u_{yy}u_y\bigr)\\
&=(p-2)(p-4)(u_y)^{p-3}u_{yy}+(p-2)(u_y)^{p-3}u_{yy}\\
&\quad +q(u_y)^{q-1} +2(p-2)(u_y)^{p-3}u_{yy}\\
&=(p-2)(p-1)(u_y)^{p-3}u_{yy}+q(u_y)^{q-1}=(q+2-p)(u_y)^{q-1}\\
&\ge (q+2-p)(\mu/2)^{q-1}>0.
\end{align*}
Therefore, owing to (\ref{uC21}), there exists $\eta>0$ such that
$$B_2(x,y,t)\ge 0 \quad\hbox{ on $(0,\hat X_1)\times[0,\eta]\times [T/4,3T/4]$},$$
which implies
\begin{equation}\label{hypcorner3}
B_2\ge -Cy\quad\hbox{ in $\overline{\hat D_T}$}.
\end{equation}

In view of (\ref{hypcorner1})-(\ref{hypcorner3}),
we may thus apply Proposition~\ref{CornerLemma} to deduce
$$
u_x(x,y,T/2)\leq -c_3 xy \quad\text{in $[0,x_1]\times[0,L_2]$.}
$$
Let $C:=\|\nabla u(\cdot,T/2)\|_\infty$. Since $\alpha>\gamma+1$, we get that, for $k\in(0,k_2]$ with $k_2>0$ small enough,
$$
J(x,y,T/2)\leq -c_3 xy+kx C^{\alpha}y^{\alpha-\gamma}\leq 
 [k C^{\alpha}L_2^{\alpha-\gamma-1}-c_3]xy\leq 0 \quad\text{in $[0,x_1]\times[0,L_2]$.}
$$
\end{proof}

\subsection{Proof of Theorem \ref{singlegbu}(ii)}

Let $\alpha,\gamma, y_1, k_0$ be given by Proposition~\ref{PJneg} and let $k_1, k_2$
be given by Lemmas~\ref{lemBoundaryCondJ}-\ref{lemInitCondJ}. We take $k=\min(k_0,k_1,k_2)$.
By these results and the maximum principle,
 we have
	\begin{equation}\label{jestneg}
	J\leq 0 \qquad\text{in } D\times (T/2, T).
	\end{equation}
Integrating inequality (\ref{jestneg}) over $(0,x)$ for $0<x<x_1$, with fixed $y$, we get that
$$u\leq Cx^{-2/(\alpha-1)}y^{1-2\sigma}\qquad\text{in $D\times(T/2,T)$},$$
where $C=C(\alpha, k, \sigma)>0$. Using that $1-2\sigma>\dfrac{q-p}{q-p+1}$,  it follows from 
the nondegeneracy property in Lemma~\ref{milouar1} that no point $(x_0,0)$ with $0<|x_0|\le \rho$ 
can be a gradient blow-up point. In view of (\ref{don0}), we conclude that $GBUS(u_0)=\{(0,0)\}$.
\hfill $\square$

	\section{Proof of the main parabolic inequality (Proposition~\ref{PJneg})}
	The proof is quite technical.
For sake of clarity, 
	some of the intermediate calculations will be summarized in Lemma~\ref{lem21} and ~\ref{lem22} below.
	
	We first compute
	\begin{eqnarray*}
J_t&=&u_{xt}+cd F'(u) u_t\\
   &=& \left(\Delta_p u\right)_x+\left(|\nabla u|^q\right)_x \underbrace{+cd F'\Delta_p u}_{(0_p)}+\hypertarget{m1}{\underbrace{cd F'|\nabla u|^q}_{ (0_q)}}.
	\end{eqnarray*}
	and
	\begin{eqnarray*}
	\left(\Delta_p u\right)_x&=&|\nabla u|^{p-2} \Delta(u_x)\\
	&+& (p-2) \Delta u |\nabla u|^{p-4} \nabla u\cdot\nabla u_x\\
	&+& (p-2) |\nabla u|^{p-4}\left\langle D^2 u_x \, \nabla u, \nabla u\right\rangle\\
	&+& (p-2)(p-4) |\nabla u|^{p-6} \nabla u\cdot\nabla u_x\left\langle D^2 u\, \nabla u, \nabla u\right\rangle\\
	&+&2(p-2) |\nabla u|^{p-4}\left\langle D^2 u\, \nabla u, \nabla u_x\right\rangle.
	\end{eqnarray*}
Using that $u_x=J- c dF(u)$, we write
$$\nabla u_x=\nabla J-cdF'\nabla u-F\begin{pmatrix} c'd\\d'c\end{pmatrix},$$
\begin{eqnarray*}
D^2 u_x&=&D^2J-cdF'D^2u-cdF''\begin{pmatrix}u_x^2&u_x u_y\\u_x u_y&u_y^2\end{pmatrix}\\
&&-F(u)\begin{pmatrix} c''d&c'd'\\c'd'&d''c\end{pmatrix}\\
&&-F'(u)\begin{pmatrix}2c'du_x&cd'u_x+c'du_y\\cd'u_x+c'du_y&2cd'u_y\end{pmatrix},\\
\end{eqnarray*}
\begin{eqnarray*}
\Delta u_x&=&Trace(D^2 u_x)= \Delta J-cdF'\Delta u-cdF''|\nabla u|^2-F[c''d+d''c] \\
&&-2F'c'dJ+2F'Fc'cd^2-2F'(u) d'c u_y,
\end{eqnarray*}
\begin{eqnarray*}
\left\langle D^2 u_x, \nabla u, \nabla u\right\rangle&=&\left\langle D^2 J, \nabla u, \nabla u\right\rangle-cdF''|\nabla u|^4-cdF'\left\langle D^2u, \nabla u, \nabla u\right\rangle-Fc''d u_x^2\\
&&-2Fc'd' u_x u_y-Fcd''u_y^2\\
&&- 2F'|\nabla u|^2(cd'u_y+dc'u_x)
\end{eqnarray*}
and
\begin{eqnarray*}
\left\langle D^2 u, \nabla u, \nabla u_x\right\rangle&=&\left\langle D^2u, \nabla u, \nabla J\right\rangle-cdF' \left\langle D^2u, \nabla u, \nabla u\right\rangle\\
&&-Fc'd \nabla u \cdot \nabla J+cc'd^2 F'F|\nabla u|^2 +d^2F^2(c')^2J\\
&&-d^2F^2(c')^2cdF+c'c d' d F^2u_y -Fd'c (u_x u_{xy}+u_y u_{yy}).
\end{eqnarray*}
 Therefore,
\begin{eqnarray*}
	\left(\Delta_p u\right)_x&=&|\nabla u|^{p-2} \left[\Delta J-cdF'\Delta u-cdF''|\nabla u|^2-2c'dF'J\right.\\
	&&\left.+2cc'd^2 F F'-2d'cF'u_y-F\left(c''d+d''c\right)\right] \\
	&&\\
	&+& (p-2) \Delta u |\nabla u|^{p-4} \left[\nabla u\cdot\nabla J-cdF'|\nabla u|^2-u_y cd'F-c'dFJ+c'cd^2F^2\right]\\
	&&\\
	&+& (p-2) |\nabla u|^{p-4}\left[\left\langle D^2 J \nabla u, \nabla u\right\rangle-cdF'\left\langle D^2 u \nabla u, \nabla u\right\rangle-cdF'' |\nabla u|^4\right.\\
	&&-\left.2c'd F' |\nabla u|^2J + 2c'cd^2 F'F |\nabla u|^2-2cd'F'u_y|\nabla u|^2\right.\\
&&-\left.c''d F (u_x)^2-2c'd'F Ju_y+2c'cd'dF^2  u_y-d''cF (u_y)^2\right]\\	
&&\\
	&+& (p-2)(p-4) |\nabla u|^{p-6}  \left\langle D^2\, u \nabla u, \nabla u\right\rangle\left[\nabla u\cdot\nabla J\right.\\
	&&-\left.cdF'|\nabla u|^2-u_y cd'F-c'dFJ+c'cd^2F^2\right]\\
	&&\\
	&+&2(p-2) |\nabla u|^{p-4}\left[\left\langle D^2 u\, \nabla u, \nabla J\right\rangle-cd F'\left\langle D^2u \, \nabla u, \nabla u\right\rangle-c'dF\nabla u\cdot\nabla J \right. \\
	&&+\left.c'cd^2F'F|\nabla u|^2+(c'd)^2F^2J-(c'd)^2F^3cd+c'd'cdF^2 u_y -d'cF u_{yy}u_y \right.\\
&&-\left.d'cF \nabla J\cdot L+d'dc^2F'Fu_yJ-d'd^2c^3F'F^2u_y+(d'cF)^2J-(d'cF)^2cdF\right],
	\end{eqnarray*}
	where $L=\begin{pmatrix}0\\u_x\end{pmatrix}$.
This can be rewritten as
	\begin{eqnarray*}
	\left(\Delta_p u\right)_x&=&|\nabla u|^{p-2}\Delta J+(p-2)|\nabla u|^{p-4}\left\langle D^2J\,\nabla u,\nabla u\right\rangle+\mathcal{H}_1\cdot\nabla J+\mathcal{A}_1(x,y,t)J	\\
&&\hypertarget{p1}{}\\
	&&\left.
\begin{array}{lll}-F|\nabla u|^{p-2}\left[c''d+d''c\right]-(p-2)F|\nabla u|^{p-4}\left[c''du_x^2+d''cu_y^2\right]\\   
	-2(p-2)cdF|\nabla u|^{p-4}\left[(c'dF)^2+(d'cF)^2\right]-(p-1)cdF''|\nabla u|^p\\
	+4(p-2)F^2c'cd'd|\nabla u|^{p-4}u_y
\end{array}\right\}\quad(1)\leq 0\\
	&&\hypertarget{p2}{}\\
	&&\left.\begin{array}{ll}+(4p-6)c'cd^2F'F|\nabla u|^{p-2}-2(p-1)cd'F'|\nabla u|^{p-2}u_y\\
	-2(p-2)d'd^2F'F^2c^3|\nabla u|^{p-4}u_y\end{array}\right\}\quad(2)\geq 0\\
&&\underbrace{- (p-1)cdF'\Delta_p u}_{(3)}\\
	&&\underbrace{-2(p-2)d'cF|\nabla u|^{p-4}u_y u_{yy}}_{(4)}\\
   	&&\underbrace{-(p-2)cd'Fu_y\left[|\nabla u|^{p-4}\Delta u+ (p-4) |\nabla u|^{p-6}  \left\langle D^2\, u \nabla u, \nabla u\right\rangle\right]}_{(5)}\\
	&&\underbrace{+(p-2)c'cd^2F^2\left[|\nabla u|^{p-4}\Delta u+ (p-4) |\nabla u|^{p-6}  \left\langle D^2\, u \nabla u, \nabla u\right\rangle\right]}_{(6)},
	\end{eqnarray*}
	\noindent where 
	\begin{eqnarray}
	\mathcal{H}_1&:=&(p-2)\left[|\nabla u|^{p-4}\Delta u+ (p-4) |\nabla u|^{p-6}  \left\langle D^2\, u \nabla u, \nabla u\right\rangle\right]\nabla u
	\nonumber \\
	&&-2(p-2)cd'F|\nabla u|^{p-4}L\nonumber \\
	&&-2(p-2)c'dF|\nabla u|^{p-4}\nabla u+2(p-2)|\nabla u|^{p-4}\left(D^2 u,\nabla u\right) \label{defH1}
	\end{eqnarray}	
	and
	\begin{eqnarray}
	\mathcal{A}_1
	&:=&-2(p-1)F'c'd|\nabla u|^{p-2} \nonumber \\
	&&-(p-2)Fc'd\left[|\nabla u|^{p-4}\Delta u+ (p-4) |\nabla u|^{p-6}  \left\langle D^2\, u \nabla u, \nabla u\right\rangle\right]\nonumber \\
	&&+2(p-2)|\nabla u|^{p-4}F^2\left[(c'd)^2+(d'c)^2\right]+2(p-2)d'dF'Fc^2|\nabla u|^{p-4}u_y\nonumber \\
	&&-2(p-2)c'd'F|\nabla u|^{p-4}u_y. \label{defA1}
	\end{eqnarray}
On the other hand, we have 
\begin{eqnarray*}
	\left(|\nabla u|^q\right)_x&=& q|\nabla u|^{q-2} \nabla u\cdot\nabla u_x\\\hypertarget{m71}{}
	&=& q|\nabla u|^{q-2} \nabla u\cdot\nabla J - qc'd F|\nabla u|^{q-2} J+(7),
	\end{eqnarray*}
	where
$$	(7):=\underbrace{-q cd F'|\nabla u|^q}_{(7_-)\leq 0}\\\hypertarget{m72}{}
	+ \underbrace{qc c' d^2 F^2 |\nabla u|^{q-2}-qcd' F |\nabla u|^{q-2} u_y}_{(7_+)\geq 0}.
	$$
Setting
 \begin{equation}
 \label{eqdefA2}
 \mathcal{A}_2:=\mathcal{A}_1- qc'd F|\nabla u|^{q-2},
\qquad
\mathcal{H}_2:=\mathcal{H}_1+q|\nabla u|^{q-2} \nabla u,
 \end{equation}
we have thus proved the following lemma.

\begin{lem}\label{lem21}
 Define the parabolic operator:
\begin{equation*}
 \mathcal{L}J:=J_t-|\nabla u|^{p-2} \Delta J-(p-2)|\nabla u|^{p-4}\left\langle D^2J\,\nabla u,\nabla u\right\rangle-\mathcal{H}_2\cdot\nabla J- \mathcal{A}_2 J.
\end{equation*}
Then
 \begin{equation}
 \label{eqLJ}
  \mathcal{L}J= (0_p)+(0_q)+(1)+(2)+(3)+(4)+(5)+(6)+(7).
 \end{equation}
\end{lem}

As a significant difficulty as compared with the semilinear case $p=2$, many additional terms appear in the 
contributions (1), (2), (4)--(6), and especially nonlinear, second order terms in (4)--(6). 
To proceed further, we need to observe that, among the second derivatives of $u$,  $u_{yy}$ needs a special treatment,
 since it is not immediately expressed in terms of $\nabla J$ unlike $u_{xx}$ and $u_{xy}$. Namely we shall  eliminate $u_{yy}$
by expressing it in terms of $u_t$, $\nabla u$, $u_{xx}$ and $u_{xy}$ by using the equation.
Although this will make the computation even more involved, by producing a lot of additional terms,
this seems to be the only way to control the effects of $u_{yy}$.
The  bound on $u_t$ given by Lemma~\ref{estimderivtemps} will be helpful in this process.\\

First  we have 
\begin{equation}
\label{eq3tq}
 (3)=- (p-1)cdF'\Delta_p u=\underbrace{-cdF'\Delta_p u}_{- (0_p)}\hypertarget{m2}{\underbrace{-(p-2) cdF'u_t}_{(3_t)}}\hypertarget{m3}{\underbrace{+(p-2) cdF' |\nabla u|^q.}_{(3_q)}}
\end{equation}
\newline
 To deal with $(4)$, we set
\begin{eqnarray*}
 u_{yy}&=&\dfrac{u_t-|\nabla u|^q-\nabla u_x\cdot M}{w},
\end{eqnarray*} where
 $$M:=\begin{pmatrix}|\nabla u|^ {p-2}+(p-2)|\nabla u|^{p-4}u_x^2\\
2(p-2)|\nabla u|^{p-4}u_x u_y\end{pmatrix}\qquad\text{and}\qquad w=|\nabla u|^{p-2}+(p-2)|\nabla u|^{p-4}u_y^2.$$
Since $u_x=J-cdF$, we get
\begin{eqnarray*}
 \nabla u_x\cdot M&=&-cdF'J\left[|(p-1)\nabla u|^{p-2}+(p-2)|\nabla u|^{p-4}u_y^2\right]-2(p-2)cd'Fu_y|\nabla u|^{p-4}J\\
&&+2(p-2)c^2d'dF^2u_y|\nabla u|^{p-4}+c^2d^2F'F\left[|(p-1)\nabla u|^{p-2}+(p-2)|\nabla u|^{p-4}u_y^2\right]\\
&&+\nabla J\cdot M-c'dF\left[|\nabla u|^{p-2}+(p-2)|\nabla u|^{p-4}u_x^2\right].
\end{eqnarray*}
It follows that 
\begin{eqnarray*}
 u_{yy}&=&\dfrac{u_t-|\nabla u|^q}{w}-\dfrac{\nabla J\cdot M}{w}+\dfrac{cdF'J\left[(p-1)|\nabla u|^{p-2}+(p-2)|\nabla u|^{p-4}u_y^2\right]}{w}\\
&&+\dfrac{2(p-2)cd'Fu_y|\nabla u|^{p-4}J}{w}
+\dfrac{c'dF\left[|\nabla u|^{p-2}+(p-2)|\nabla u|^{p-4}u_x^2\right]}{w} \\
&&-\dfrac{c^2d^2FF'\left[(p-1)|\nabla u|^{p-2}+(p-2)|\nabla u|^{p-4}u_y^2\right]}{w}\\
&&-\dfrac{2(p-2)c^2d'dF^2u_y|\nabla u|^{p-4}}{w}.
\end{eqnarray*}
Now, to treat  the contribution of $u_{yy}$ in $(5)$ and $(6)$, we set $N=\begin{pmatrix} u_x^2\\2u_xu_y\end{pmatrix}$ and rewrite
\begin{eqnarray*} 
|\nabla u|^{p-4}\Delta u+(p-4)|\nabla u|^{p-6}\left\langle D^2 u\,  \nabla u, \nabla u\right\rangle&=&\dfrac{\Delta_p u}{|\nabla u|^2}
-2 |\nabla u|^{p-6}\left\langle D^2 u\, \nabla u, \nabla u\right\rangle\\
&=&\dfrac{u_t-|\nabla u|^q}{|\nabla u|^2}-2|\nabla u|^{p-6}\left[ \nabla u_x\cdot N+u_y^2 u_{yy}\right].
\end{eqnarray*}
We have
\begin{eqnarray*}
 \nabla u_x \cdot N&=&\nabla J\cdot N-cdF'J\left[u_x^2+2u_y^2\right]
+ c^2d^2 F F'\left[u_x^2+2u_y^2\right]
-c'd Fu_x^2\\
&&-2cd'F Ju_y+2d'dc^2F^2u_y.
\end{eqnarray*}
The expression in $(4)$ then becomes
\begin{eqnarray*}
(4)&=&\left.\dfrac{2(p-2)d'cF|\nabla u|^{p-4}u_y  M\cdot \nabla J}{w}\right\} (4_\nabla)\\\\ 
&&\left.\begin{array}{ll}-\dfrac{2(p-2)c^2d'dF'F|\nabla u|^{p-4}u_y\left[(p-1)|\nabla u|^{p-2}+(p-2)|\nabla u|^{p-4}u_y^2\right]J}{w}\\
-\dfrac{4(p-2)^2(d'cF)^2|\nabla u|^{p-4}u_y^2|\nabla u|^{p-4}J}{w}
\end{array}\right\}(4_J)\\
&&\hypertarget{m41}{}\\
&&\left.\begin{array}{ll}+\dfrac{2d'cF(p-2)|\nabla u|^{p-4}|\nabla u|^qu_y}{w}\\
+\dfrac{2(p-2)c^3d'd^2F'F^2|\nabla u|^{p-4}u_y\left[(p-1)|\nabla u|^{p-2}+(p-2)|\nabla u|^{p-4}u_y^2\right]}{w}\end{array}\right\}(4_-)\leq 0\\
&&\hypertarget{m42}{}\\
&&\left.\begin{array}{ll}+\dfrac{4(p-2)^2(d'cF)^2cdF|\nabla u|^{p-4}u_y^2|\nabla u|^{p-4}}{w}\\
-\dfrac{2(p-2)d'dc'cF^2u_y|\nabla u|^{p-4}\left[|\nabla u|^{p-2}+(p-2)|\nabla u|^{p-4}u_x^2\right]}{w}\end{array}\right\}(4_+)\geq 0\\
&&\left.-\dfrac{2d'cF(p-2)|\nabla u|^{p-4}u_tu_y}{w}\right\}(4_t)\hypertarget{m43}{}.
\end{eqnarray*}
The other two terms can be rewritten as
\begin{eqnarray*}
(5)&=&\left.2cd'F(p-2)u_y|\nabla u|^{p-6}  N\cdot \nabla J-\dfrac{2(p-2)cd'Fu_y^3|\nabla u|^{p-6} M\cdot \nabla J}{w}\right\}(5_\nabla)\\\\
&&\left.\begin{array}{ll}-2(p-2)c^2d'dF'Fu_y|\nabla u|^{p-6}\left[u_x^2+2u_y^2\right]J-4(p-2)(cd'F)^2u_y^2|\nabla u|^{p-6}J\\
\noalign{\vskip 1mm}
+\dfrac{2(p-2)c^2d'dF'Fu_y^2|\nabla u|^{p-6}u_y\left[(p-1)|\nabla u|^{p-2}+(p-2)|\nabla u|^{p-4}u_y^2\right]J}{w}\\
+\dfrac{4(p-2)^2(d'cF)^2u_y^4|\nabla u|^{ 2p-10}J}{w}\\ 
\end{array}\right\}(5_J)\\
\noalign{\vskip 1mm}
&&\hypertarget{m51}{}\\
&&\left.
\begin{array}{lll}+(p-2)cd'Fu_y|\nabla u|^{q-2}+2(p-2)c^3d'd^2F'F^2u_y|\nabla u|^{p-6}\left[u_x^2+2u_y^2\right]\\
\noalign{\vskip 1mm}
+\dfrac{2(p-2)c'cd'dF^2u_y|\nabla u|^{p-6}u_y^2\left[|\nabla u|^{p-2}+(p-2)|\nabla u|^{p-4}u_x^2\right]}{w}\\
\noalign{\vskip 1mm}
-\dfrac{4(p-2)^2(cd'F)^2cdFu_y^4|\nabla u|^{2p-10}}{w}\end{array}\right\}(5_-)\leq 0 \\
\noalign{\vskip 1mm}
&&\hypertarget{m52}{}\\
&&\left.\begin{array}{lll}-2(p-2)c'cd'dF^2u_y|\nabla u|^{p-6}u_x^2+4(p-2)(cd'F)^2cdFu_y^2|\nabla u|^{p-6}\\
\noalign{\vskip 1mm}
-\dfrac{2(p-2)cd'Fu_y|\nabla u|^{p-4}|u_y^2|\nabla u|^{q-2}}{w}\\
\noalign{\vskip 1mm}
-\dfrac{2(p-2)c^3d'd^2F'F^2u_y|\nabla u|^{p-6}u_y^2\left[(p-1)|\nabla u|^{p-2}+(p-2)|\nabla u|^{p-4}u_y^2\right]}{w}\end{array}\right\}(5_+)\geq 0\\
&&\hypertarget{m53}{}\\
&&\left.-\dfrac{(p-2)cd'Fu_yu_t}{|\nabla u|^2}+\dfrac{2(p-2)cd'Fu_y^2|\nabla u|^{p-6}u_t u_y}{w}\right\}(5_t)
\end{eqnarray*}
 and (noticing that (6) can be obtained from (5) by formally multiplying with $\frac{-c'd^2F}{d'u_y}$)
\begin{eqnarray*}
(6)&=&\left.-2(p-2)c'cd^2F^2
|\nabla u|^{p-6} { N\cdot \nabla J}+ \dfrac{2(p-2)c'cd^2F^2u_y^2|\nabla u|^{p-6}  M\cdot \nabla J}{w}
\right\}(6_\nabla)\\\\
&&\left.\begin{array}{ll}+2(p-2)c'c^2d^3F'F^2|\nabla u|^{p-6}\left[u_x^2+2u_y^2\right]J+4(p-2)c'c^2d'd^2F^3u_y|\nabla u|^{p-6}J\\
-\dfrac{2(p-2)c'c^2d^3F'F^2u_y^2|\nabla u|^{p-6}\left[(p-1)|\nabla u|^{p-2}+(p-2)|\nabla u|^{p-4}u_y^2\right]J}{w}\\
-\dfrac{4(p-2)^2c'c^2d'd^2F^3|\nabla u|^{p-6}u_y^3|\nabla u|^{p-4}J}{w}
\end{array}\right\} (6_J)\\
&&\hypertarget{m61}{}\\
&&\left.\begin{array}{lll}-(p-2)c'cd^2F^2|\nabla u|^{q-2}-2(p-2)c'c^3d^4F'F^3|\nabla u|^{p-6}\left[u_x^2+2u_y^2\right]\\
-\dfrac{2(p-2)(c')^2cd^3F^3|\nabla u|^{p-6}u_y^2\left[|\nabla u|^{p-2}+(p-2)|\nabla u|^{p-4}u_x^2\right]}{w}\\
+\dfrac{4(p-2)^2c'c^3d'd^3F^4u_y^3|\nabla u|^{2p-10}}{w}\end{array}\right\}(6_-)\leq 0\\
&&\hypertarget{m62}{}\\
&&\left.\begin{array}{lll}+2(p-2)(c')^2cd^3F^3|\nabla u|^{p-6}u_x^2-4(p-2)c'c^3d'd^3F^4u_y|\nabla u|^{p-6}\\
+\dfrac{2(p-2)c'cd^2F^2|\nabla u|^{p-4}u_y^2|\nabla u|^{q-2}}{w}\\
+\dfrac{2(p-2)c'c^3d^4F'F^3|\nabla u|^{p-6}u_y^2\left[(p-1)|\nabla u|^{p-2}+(p-2)|\nabla u|^{p-4}u_y^2\right]}{w}\end{array}\right\}(6_+)\geq 0\\
&&\hypertarget{m63}{}\\
&&\left.+\dfrac{(p-2)c'cd^2F^2u_t}{|\nabla u|^2}-\dfrac{2(p-2)c'cd^2F^2u_y^2|\nabla u|^{p-6}u_t}{w}.\right\}(6_t)
\end{eqnarray*}
 We shall now collect and relabel the numerous positive and negative terms 
that we just obtained, when expanding (1)--(7) in the process of eliminating $u_{yy}$.
A number of positive and negative terms will then be paired together according to certain cancellations.
Then, the remaining positive terms, as well as the terms involving $u_t$, will be 
eventually controlled by using the negative terms.

Using that $d'\leq 0$ and $ F', F'', u_y\geq 0$, 
we  first have positive terms:
\begin{eqnarray*}
&&(a):=-2(p-2)c^3d'd^2F'F^2|\nabla u|^{p-4}u_y\\
&&(b):=-\dfrac{2(p-2)cd'Fu_y^2|\nabla u|^{p-6}|\nabla u|^q u_y}{w}\\
&&(c):=-\dfrac{2(p-2)c^3d'd^2F'F^2u_y|\nabla u|^{p-6}u_y^2\left[(p-1)|\nabla u|^{p-2}+(p-2)|\nabla u|^{p-4}u_y^2\right]}{w}\\
&&(d):=+2(p-2)(c')^2cd^3F^3|\nabla u|^{p-6}u_x^2\\
&&(e):=+\dfrac{2(p-2)c'c^3d^4F'F^3|\nabla u|^{p-6}u_y^2\left[(p-1)|\nabla u|^{p-2}+(p-2)|\nabla u|^{p-4}u_y^2\right]}{w}\\
&&(f):=\left.-\dfrac{2(p-2)d'dc'cF^2u_y|\nabla u|^{p-4}\left[|\nabla u|^{p-2}+(p-2)|\nabla u|^{p-4}u_x^2\right]}{w}\right\}(f_1)\\
&&\qquad\left.-2(p-2)c'cd'dF^2u_y|\nabla u|^{p-6}u_x^2\right\}(f_2)
\end{eqnarray*}
\begin{eqnarray*}
&&(g):=-4(p-2)c'c^3d'd^3F^4u_y|\nabla u|^{p-6}\\
&&(h):=\underbrace{+\dfrac{4(p-2)^2(d'cF)^2cdF|\nabla u|^{p-4}u_y^2|\nabla u|^{p-4}}{w}}_{(h_1)}+\underbrace{4(p-2)(cd'F)^2cdFu_y^2|\nabla u|^{p-6}}_{(h_2)}\\
&&(i):=-qd'cF|\nabla u|^{q-2}u_y\\
&&(j):=\underbrace{+qc'cd^2F^2|\nabla u|^{q-2}}_{(j_1)}\underbrace{+\dfrac{2(p-2)c'cd^2F^2|\nabla u|^{p-4}u_y^2|\nabla u|^{q-2}}{w}}_{(j_2)}\\
&&(l):=-2(p-1)cd'F'|\nabla u|^{p-2}u_y\\
&&(m):=+(4p-6)cc'd^2F'F|\nabla u|^{p-2}.
\end{eqnarray*}
 They give rise to the following decompositions:
 \begin{equation}
 \label{eqDecompPlus}
\left\{\begin{array}{lll}
\hyperlink{p2}{(2)}&=&(m)+(l)+(a)\\ \noalign{\vskip 2mm}
\hyperlink{m42}{(4_+)}&=&(f_1)+(h_1) \\ \noalign{\vskip 2mm}
\hyperlink{m52}{(5_+)}&=&(b)+(c) +(h_2)+(f_2)\\ \noalign{\vskip 2mm}
\hyperlink{m62}{(6_+)}&=& (d)+(e)+(g)+(j_2)\\ \noalign{\vskip 2mm}
\hyperlink{m72}{(7_+)}&=&(i)+(j_1).\\
\end{array}\right.
 \end{equation}
 We next have terms with a negative sign:
\begin{eqnarray*}
&&(\tilde{a}):=+2(p-2)c^3d'd^2F'F^2u_y|\nabla u|^{p-6}\left[u_x^2+2u_y^2\right]\\
&&(\tilde{b}):=+\dfrac{2d'cF(p-2)|\nabla u|^{p-4}|\nabla u|^qu_y}{w}\\
&&(\tilde{c}):=+\dfrac{2(p-2)c^3d'd^2F'F^2|\nabla u|^{p-4}u_y\left[(p-1)|\nabla u|^{p-2}+(p-2)|\nabla u|^{p-4}u_y^2\right]}{w}\\
&&(\tilde{d}):=-2(p-2)cdF|\nabla u|^{p-4}(c'dF)^2\\
&&(\tilde{e}):=-2(p-2)c'c^3d^4F'F^3|\nabla u|^{p-6}\left[u_x^2+2u_y^2\right]\\
&&(\tilde{f}):=\left.\dfrac{2(p-2)c'cd'dF^2u_y|\nabla u|^{p-6}u_y^2\left[|\nabla u|^{p-2}+(p-2)|\nabla u|^{p-4}u_x^2\right]}{w}\right\}(\tilde{f}_1)\\
&&\qquad\left. +4(p-2)c'cd'dF^2|\nabla u|^{p-4}u_y\right\}(\tilde{f}_2)
\end{eqnarray*}
\begin{eqnarray*}
&&(\tilde{g}):=+(p-2)d'cF|\nabla u|^{q-2}u_y\\
&&(\tilde{h}):=\underbrace{-(p-1)cdF''|\nabla u|^p}_{(\tilde{h}_1)}\underbrace{+(p-1-q)|\nabla u|^qcdF'}_{(\tilde{h}_2)}\underbrace{-(p-1)cd''F|\nabla u|^{p-4}u_y^2}_{(\tilde{h}_3)}\\
&&\\
&&
\begin{array}{llll}
(\tilde{i}):=\underbrace{-cd''F|\nabla u|^{p-4}u_x^2}_{(\tilde{i}_1)}
 \underbrace{-(p-2)c'cd^2F^2|\nabla u|^{q-2}}_{(\tilde{i}_1)}
\underbrace{+\dfrac{4(p-2)^2c'c^3d'd^3F^4u_y^3|\nabla u|^{2p-10}}{w}}_{(\tilde{i}_3)}\\
(\tilde{j}):=\underbrace{-2(p-2)(cd'F)^2cdF|\nabla u|^{p-4}}_{(\tilde{j}_1)}
\underbrace{-\dfrac{4(p-2)^2(cd'F)^2cdFu_y^4|\nabla u|^{2p-10}}{w}}_{(\tilde{j}_2)}\\
(\tilde{l}):=-\dfrac{2(p-2)(c')^2cd^3F^3|\nabla u|^{p-6}u_y^2\left[|\nabla u|^{p-2}+(p-2)|\nabla u|^{p-4}u_x^2\right]}{w}.
\end{array}
\end{eqnarray*}
With these terms, we have the following decompositions (using $c''=0$):
 \begin{equation}
 \label{eqDecompMinus}
 \left\{\begin{array}{lll}
\hyperlink{p1}{(1)}=(\tilde{d})+(\tilde{f}_2)+(\tilde{h}_1)+(\tilde{h}_3)+(\tilde{i}_1)+(\tilde{j}_1)\\ \noalign{\vskip 2mm}
\hyperlink{m41}{(4_-)}=(\tilde{b})+(\tilde{c})\\ \noalign{\vskip 2mm}
\hyperlink{m51}{(5_-)}=(\tilde{a})+(\tilde{g})+(\tilde{j}_2) +(\tilde{f}_1)\\ \noalign{\vskip 2mm}
\hyperlink{m61}{(6_-)}= (\tilde{e})+(\tilde{i}_2)+(\tilde{i}_3)+(\tilde{l})\\ \noalign{\vskip 2mm}
 \hyperlink{m71}{(7_-)}=(\tilde{h}_2)-\hyperlink{m3}{(0_q)}- \hyperlink{m1}{(3_q)}.\\
\end{array}\right. 
\end{equation}
It follows from (\ref{eqLJ}) in Lemma~\ref{lem21} and (\ref{eq3tq})--(\ref{eqDecompMinus}) that
$$
\aligned 
\mathcal{L}J= 
& \phantom{ +}\bigl[(0_q)+(0_p)\bigr]  \\
& +\bigl[(\tilde{d})+(\tilde{f}_2)+(\tilde{h}_1)+(\tilde{h}_3)+(\tilde{i}_1)+(\tilde{j}_1)\bigr] \\
& + \bigl[(m)+(l)+(a)\bigr] \\
& +\bigl[- (0_p)+(3_t)+(3_q)\bigr] \\
& + \bigl[(\tilde{b})+(\tilde{c})\bigr] +\bigl[(f_1)+(h_1) \bigr] + \bigl[(4_\nabla) + (4_J) + (4_t)\bigr] \\
& +\bigl[(\tilde{a})+(\tilde{g})+(\tilde{j}_2)+(\tilde{f}_1)\bigr] + \bigl[(b) +(c) +(h_2)+(f_2)\bigr] + \bigl[(5_\nabla) + (5_J) + (5_t)\bigr]  \\
& +\bigl[(\tilde{e})+(\tilde{i}_2)+(\tilde{i}_3)+(\tilde{l})\bigr] + \bigl[(d)+(e)+(g)+(j_2)\bigr]  + \bigl[(6_\nabla) + (6_J) + (6_t) \bigr] \\
&+\bigl[ (\tilde{h}_2)-\hyperlink{m3}{(0_q)}- \hyperlink{m1}{(3_q)}\bigr] + \bigl[    (i)+(j_1)\bigr]. \\
\endaligned
$$
Reordering the terms, we obtain
\begin{equation}
\label{eqLJ2}
\aligned 
\mathcal{L}J
=& \phantom{+}(a)+(b)+(c)+(d)+(e)+(f)+(g)+(h)+(i)+(j)+(l)+(m) \\
&+(\tilde a)+(\tilde b)+(\tilde c)+(\tilde d)+(\tilde e)+(\tilde f)+(\tilde g)+(\tilde h)+(\tilde i)+(\tilde j)+(\tilde l) \\
&+\bigl[(3_t) +(4_t) + (5_t) + (6_t)\bigr] \\
&+\bigl[(4_\nabla) +(5_\nabla) +(6_\nabla) \bigr]
+\bigl[ (4_J) + (5_J) +(6_J) \bigr].
\endaligned
\end{equation}
Collecting the terms with $J$ (reps., $\nabla J$) in (\ref{eqLJ2}), together with those in $\mathcal{A}_2$
(resp., $\mathcal{H}_2$) and using (\ref{eqdefA2}), we define
\begin{eqnarray}
 \mathcal{H}&:=&\mathcal{H}_1+q|\nabla u|^{q-2} \nabla u+\dfrac{2(p-2)d'cF|\nabla u|^{p-4}u_y M}{w} \nonumber\\
&&\phantom{\mathcal{H}_2}+2cd'F(p-2)u_y|\nabla u|^{p-6} N
-\dfrac{2(p-2)cd'Fu_y^3|\nabla u|^{p-6} M}{w} \label{defH}\\
&&\phantom{\mathcal{H}_2} -2(p-2)c'cd^2F^2|\nabla u|^{p-6} N + \dfrac{2(p-2)c'cd^2F^2u_y^2|\nabla u|^{p-6} M}{w} \nonumber
\end{eqnarray}
and
\begin{eqnarray}
\mathcal{A}
&:= &\mathcal{A}_1  - qc'd F|\nabla u|^{q-2}\nonumber\\
\noalign{\vskip 1mm}
&&\phantom{\mathcal{A}_1}-\dfrac{2(p-2)c^2d'dF'F|\nabla u|^{p-4}u_y\left[(p-1)|\nabla u|^{p-2}+(p-2)|\nabla u|^{p-4}u_y^2\right]}{w}\nonumber\\
&&\phantom{\mathcal{A}_1}  -\dfrac{4(p-2)^2(d'cF)^2|\nabla u|^{p-4}u_y^2|\nabla u|^{p-4}}{w}\nonumber\\
&&\phantom{\mathcal{A}_1}  -2(p-2)c^2d'dF'Fu_y|\nabla u|^{p-6}\left[u_x^2+2u_y^2\right]-4(p-2)(cd'F)^2u_y^2|\nabla u|^{p-6}\nonumber\\
&&\phantom{\mathcal{A}_1}  +\dfrac{2(p-2)c^2d'dF'Fu_y^2|\nabla u|^{p-6}u_y\left[(p-1)|\nabla u|^{p-2}+(p-2)|\nabla u|^{p-4}u_y^2\right]}{w}
\nonumber\\
&&\phantom{\mathcal{A}_1}  +\dfrac{4(p-2)^2(d'cF)^2u_y^4|\nabla u|^{2p-10}}{w} \label{defA}\\
&&\phantom{\mathcal{A}_1}  +2(p-2)c'c^2d^3F'F^2|\nabla u|^{p-6}\left[u_x^2+2u_y^2\right]+4(p-2)c'c^2d'd^2F^3u_y|\nabla u|^{p-6}
\nonumber\\
&&\phantom{\mathcal{A}_1}  -\dfrac{2(p-2)c'c^2d^3F'F^2u_y^2|\nabla u|^{p-6}
\left[(p-1)|\nabla u|^{p-2}+(p-2)|\nabla u|^{p-4}u_y^2\right]}{w}\nonumber\\
&&\phantom{\mathcal{A}_2}  -\dfrac{4(p-2)^2c'c^2d'd^2F^3|\nabla u|^{p-6}u_y^3|\nabla u|^{p-4}}{w},\nonumber
\end{eqnarray}
where we recall that
\begin{eqnarray}
	\mathcal{H}_1&:=&(p-2)\left[|\nabla u|^{p-4}\Delta u+ (p-4) |\nabla u|^{p-6}  \left\langle D^2\, u \nabla u, \nabla u\right\rangle\right]\nabla u
	\nonumber \\
	&&-2(p-2)cd'F|\nabla u|^{p-4}L\nonumber \\
	&&-2(p-2)c'dF|\nabla u|^{p-4}\nabla u+2(p-2)|\nabla u|^{p-4}\left(D^2 u,\nabla u\right), \label{defH1b}
	\end{eqnarray}	
	with $L=\begin{pmatrix}0\\u_x\end{pmatrix}$, and
	\begin{eqnarray}
	\mathcal{A}_1
	&:=&-2(p-1)F'c'd|\nabla u|^{p-2} \nonumber \\
	&&-(p-2)Fc'd\left[|\nabla u|^{p-4}\Delta u+ (p-4) |\nabla u|^{p-6}  \left\langle D^2\, u \nabla u, \nabla u\right\rangle\right]\nonumber \\
	&&+2(p-2)|\nabla u|^{p-4}F^2\left[(c'd)^2+(d'c)^2\right]+2(p-2)d'dF'Fc^2|\nabla u|^{p-4}u_y\nonumber \\
	&&-2(p-2)c'd'F|\nabla u|^{p-4}u_y. \label{defA1b}
	\end{eqnarray}
Finally observing that
$$(a)+(\tilde{a})\leq0, \quad
(b)+(\tilde{b})\leq0, \quad
(c)+(\tilde{c})\leq0, \quad
(d)+(\tilde{d})\leq 0, \quad
(e)+(\tilde{e})\leq 0$$
and using $(\tilde f), (\tilde i), (\tilde j), (\tilde l) \le 0$,  we obtain the following lemma.

\begin{lem}\label{lem22}
Recalling the definition (\ref{DefOperP}) of the parabolic operator $\mathcal {P}$:
$$\mathcal {P} J:=J_t- |\nabla u|^{p-2} \Delta J- (p-2)|\nabla u|^{p-4}\left\langle D^2J\,\nabla u,\nabla u\right\rangle
-\mathcal{H }\cdot\nabla J-\mathcal{A}(x,y,t)J,$$
we have
\begin{eqnarray*}
\mathcal{P}J&\le &(f)+(g)+(h)+(i)+(j)+(l)+(m)+(\tilde g)+(\tilde h) \\
&&+(3_t) +(4_t) + (5_t) + (6_t).\\
\end{eqnarray*}
\end{lem}

\begin{proof}[Completion of proof of Proposition~\ref{PJneg}]
 Starting from Lemma~\ref{lem22}, we shall estimate the remaining positive and $u_t$ terms by the 
key negative terms $(\tilde g)$ and $(\tilde h)$, 
after appropriate choice of the parameters.
An essential tool in this step will the Bernstein-type estimates (see \cite[Theorem~1.2]{AmalJDE})
\begin{equation}
\label{BernsteinEst}
|\nabla u|\leq C_0 y^{-1/(q-p+1)}\quad\text{and}\quad u \leq  C_0 y^{(q-p)/(q-p+1)}\quad \text{in}\  [0,x_1]\times (0,L_2]\times (0, T),
\end{equation}
where $C_0=C_0(p,q,\Omega,\mu,\|\nabla u_0\|_\infty)>0$,
and we will use also the lower bound from (\ref{don1}):
\begin{equation}
\label{lowerNabla}
|\nabla u|\geq u_y\geq \delta_0=\mu/2>0.
\end{equation}

First, using $w\ge |\nabla u|^{p-2}$, we get
$$(f)\leq -2(p-2)pc'cd'dF^2u_y|\nabla u|^{p-4}.$$
Assume $y_1\le 1$. Due to (\ref{BernsteinEst}),
we have 
\begin{equation}\label{permis}
 dF\leq C_0^{\alpha}y^{-\gamma+\alpha(q-p)/(q-p+1)} \leq   C_0^{\alpha}.
\end{equation}
Here we used that $\gamma\leq \alpha-1$ and hence $\alpha(q-p)/(q-p+1)-\gamma\geq 1-\alpha/(q-p+1)\geq 0$.
Assume $k_0=k_0(p,q,\Omega,\mu,\|\nabla u_0\|_\infty)>0$ sufficiently small so that
\begin{eqnarray}
0<k_0\leq \dfrac{|\nabla u|^{q-p+2}}{4 p dF}\qquad\text{and}\qquad
0<k_0^3\leq \dfrac{|\nabla u|^{q-p+4}}{8x^2F^3d^3},
\end{eqnarray}
which is possible due to (\ref{lowerNabla}),  
$x\leq L_1$ and \eqref{permis}).
We then have
\begin{eqnarray*}
&& (f)+(g)+(\tilde{g})\\
&\leq&\underbrace{\dfrac{p-2}{2}cd'Fu_y}_{\leq 0}\left[\underbrace{|\nabla u|^{q-2}-8k^3x^2F^3d^3 |\nabla u|^{p-6}+|\nabla u|^{q-2}-4 p k|\nabla u|^{p-4}dF}_{\geq 0}\right]\le 0.
\end{eqnarray*}

Next, we have 
\begin{eqnarray*}
(h)&\leq&4(p-2)(p-1)(cd'F)^2cdFu_y^2|\nabla u|^{p-6}\\
(j)&\leq&(q+2(p-2))c'cd^2F^2|\nabla u|^{q-2}
\end{eqnarray*}
and, owing to Lemma~\ref{estimderivtemps}, 
\begin{eqnarray*}
(3_t)&\leq& (p-2)cdF'K \\
(4_t)+(5_t)&\leq&5(p-2)c|d'|F\dfrac{|u_t|u_y}{|\nabla u|^2}\leq 5(p-2)c|d'|F\dfrac{Ku_y}{|\nabla u|^2}\\
(6_t)&\leq& \dfrac{3(p-2) c'cd^2F^2 K}{|\nabla u|^2}.
\end{eqnarray*}
Here and in the rest of the proof, $K$ denotes a constant depending on $\|u_0\|_{C^2}$, $p$ and $q$.
Consequently
\begin{eqnarray*}
\mathcal{P} J&\leq&(h)+(i)+(j)+(l)+(m)+(3_t) +(4_t) + (5_t) + (6_t)+(\tilde{h})\\
&\leq&4(p-2)(p-1)(cd'F)^2cdFu_y^2|\nabla u|^{p-6}+(q+2(p-2))c'cd^2F^2|\nabla u|^{q-2}\\
&&+5(p-2)c|d'|F\dfrac{Ku_y}{|\nabla u|^2}+2(p-1)c|d'|F'|\nabla u|^{p-2}u_y\\ 
&&+(4p-6)cc'd^2F'F|\nabla u|^{p-2}+q|d'|cF|\nabla u|^{q-2}u_y\\
&&+\dfrac{3(p-2) c'cd^2F^2 K}{|\nabla u|^2}+(p-2)cdF'K\\ 
&&-(p-1)cdF''|\nabla u|^p+(p-1-q)|\nabla u|^qcdF'-(p-1)cd''F|\nabla u|^{p-4}u_y^2
\end{eqnarray*} 
hence
\begin{eqnarray}\label{eq1a}
\dfrac{\mathcal{P}J}{cdF}&\leq&
-(q-p+1)\alpha\dfrac{|\nabla u|^q}{u}-(p-1)\alpha(\alpha-1)\dfrac{|\nabla u|^p}{u^2}-(p-1)\dfrac{\gamma (\gamma +1)|\nabla u|^{p-4}u_y^2}{y^2}
\nonumber\\
&+&\dfrac{2\gamma\alpha(p-1)}{y}\dfrac{|\nabla u|^{p-2}u_y}{u}+4(p-2)(p-1)k^2\gamma^2 x^2u^{2\alpha}y^{-2\gamma}\dfrac{|\nabla u|^{p-6}u_y^2}{y^2}\nonumber\\
&+&5(p-2)\gamma  K|\nabla u|^{-2}\dfrac{u_y}{y} 
+(4p-6)k\alpha\dfrac{u^{\alpha-1}|\nabla u|^{p-2}}{y^{\gamma}}+
(q+2(p-2))k\dfrac{u^{\alpha}|\nabla u|^{q-2}}{y^{\gamma}}\nonumber\\
&+&q\gamma\dfrac{|\nabla u|^{q-2}u_y}{y}+(p-2)\dfrac{\alpha K}{u}+\dfrac{3(p-2)k y^{-\gamma}u^{\alpha} 
K} {|\nabla u|^2}.
\end{eqnarray}

Using Young's inequality, we obtain that
$$\dfrac{2\gamma\alpha}{y}\dfrac{|\nabla u|^{p-2}u_y}{u}\leq \alpha(\alpha-1)\dfrac{|\nabla u|^{p}}{u^2}+\dfrac{\alpha\gamma^2}{\alpha -1}\dfrac{|\nabla u|^{p-4}u_y^2}{y^2},$$
hence
\begin{eqnarray}\label{eq1}
&&\dfrac{2\gamma\alpha(p-1)}{y}\dfrac{|\nabla u|^{p-2}u_y}{u}
-(p-1)\alpha(\alpha-1)\dfrac{|\nabla u|^p}{u^2}-(p-1)\dfrac{\gamma (\gamma +1)|\nabla u|^{p-4}u_y^2}{y^2}\nonumber\\
&&\leq\left(\dfrac{\alpha\gamma^2}{\alpha-1}-\gamma(\gamma+1)   \right) \dfrac{(p-1)|\nabla u|^{p-4}u_y^2}{y^2}
=-\dfrac{2\gamma\sigma(p-1)|\nabla u|^{p-4}u_y^2}{y^2}.
\end{eqnarray}
 By (\ref{BernsteinEst}), we have also
\begin{equation}\label{eq7}
u|\nabla u|^{q-p}\leq C_0^{q-p+1}.
\end{equation}
Using again Young's inequality,  and (\ref{eq7}), we have
\begin{align}
&q\gamma\dfrac{|\nabla u|^{q-2}u_y}{y}\leq\dfrac{\sigma\gamma |\nabla u|^{p-4}u_y^2}{2y^2}+\dfrac{q^2\gamma}{2\sigma}|\nabla u|^{2q-p}
 \leq\dfrac{\sigma\gamma |\nabla u|^{p-4}u_y^2}{2y^2}+\dfrac{q^2\gamma C_0^{q-p+1}}{2\sigma}\dfrac{|\nabla u|^q}{u},\label{eq2}\\
&5(p-2)\gamma K|\nabla u|^{-2}\dfrac{u_y}{y}\leq \dfrac{\sigma\gamma|\nabla u|^{p-4}u_y^2}{2y^2}+\left[\dfrac{25(p-2)^2\gamma  K^2|\nabla u|^{-p-q} u}{2\sigma}\right]\dfrac{|\nabla u|^q}{u}.\label{eq3}
\end{align}
Next, using (\ref{BernsteinEst}) and  (\ref{lowerNabla}),
it follows that 
\begin{align}
k\alpha(4p-6)\dfrac{u^{\alpha-1}|\nabla u|^{p-2}}{y^{\gamma}}&\leq k\alpha(4p-6)C_0^{\alpha -1} y^{(\alpha-1)(2\sigma-\frac{1}{q-p+1})+2}\dfrac{|\nabla u|^2}{u_y^2}\dfrac{|\nabla u|^{p-4}u_y^2}{y^2}\nonumber\\
&\leq k\alpha(4p-6)\delta_0^{-2} C_0^{\alpha +1} y^{(\alpha-1)(2\sigma-\frac{1}{q-p+1})+\frac{2(q-p)}{q-p+1}}\dfrac{|\nabla u|^{p-4}u_y^2}{y^2},\label{eq4}
\end{align}
\begin{eqnarray}\label{eq5}
&(q+2(p-2))k\dfrac{u^{\alpha}|\nabla u|^{q-2}}{y^{\gamma}}\leq (q+2(p-2))k C_0^{\alpha +q-p}y^{(\alpha-1)(2\sigma-\frac{1}{q-p+1})+2}\dfrac{|\nabla u|^2}{u_y^2}\dfrac{|\nabla u|^{p-4}u_y^2}{y^2}\nonumber\\ 
&\leq(q+2(p-2))k \dfrac{C_0^{\alpha+q-p+2}}{ \delta_0^2}y^{(\alpha-1)(2 \sigma-\frac{1}{q-p+1})+\frac{2(q-p)}{q-p+1}}\dfrac{|\nabla u|^{p-4}u_y^2}{y^2}
\end{eqnarray}
and
\begin{eqnarray}\label{eq6}
k^2\gamma^2\dfrac{ x^2u^{2\alpha}}{y^{2\gamma}}\dfrac{|\nabla u|^{p-6}u_y^2}{y^2}\leq k^2\gamma^2\dfrac{ x^2 C_0^{2\alpha}y^{-2\gamma+2\alpha(q-p)/(q-p+1)}}{\delta_0^{2}}\dfrac{|\nabla u|^{p-4}u_y^2}{y^2}.
\end{eqnarray}
Finally, using that $u\geq \mu y$, we have
\begin{equation}\label{eq8}
 \dfrac{\alpha K}{u}=\dfrac{\alpha K y^2}{u |\nabla u|^{p-4} u_y^2}\dfrac{|\nabla u|^{p-4} u_y^2}{y^2}\leq \dfrac{\alpha Ky}{\mu \delta_0^{p-2}}\dfrac{|\nabla u|^{p-4} u_y^2}{y^2}.
\end{equation}
Using the  bounds $u\le \|u_0\|_\infty$ and (\ref{lowerNabla}) we have
\begin{equation}\label{eq9}
 \dfrac{k y^{-\gamma}u^{\alpha} K}{ |\nabla u|^2}\leq \dfrac{k \|u_0\|_{\infty}^{\alpha}Ky^{2-\gamma}}{ \delta_0^p}\dfrac{|\nabla u|^{p-4} u_y^2}{y^2}.
\end{equation}
Combining  \eqref{eq1a}-\eqref{eq9},  we get that 
\begin{equation}\label{finalcalc}
\begin{array}{lllll}
\dfrac{\mathcal{P}J}{cdF}&\leq& 
\phantom{+}\dfrac{|\nabla u|^q}{u}\left(\displaystyle\frac{q^2\gamma C_0^{q-p+1}}{2\sigma}+25(p-2)^2\gamma \dfrac {K^2
\|u_0\|_{\infty}}{ 2\sigma\delta_0^{q+p}}-\alpha(q-p+1)\right)\\
&&+\dfrac{|\nabla u|^{p-4}u_y^2}{y^2} 
\biggl\{k\left((q+2(p-2)) C_0^{\alpha+q-p+2} \delta_0^{-2}\right) y^{(\alpha-1)(2\sigma-\frac{1}{q-p+1})+\frac{2(q-p)}{q-p+1}}\\
&&+k\left(\alpha(4p-6)\delta_0^{-2} C_0^{\alpha +1}\right) y^{(\alpha-1)(2\sigma-\frac{1}{q-p+1})+\frac{2(q-p)}{q-p+1}}\\
\noalign{\vskip 2mm}
&&+4(p-2)(p-1)k^2\gamma^2 x^2 C_0^{2\alpha} y^{-2\gamma+\frac{2\alpha(q-p)}{q-p+1}}\delta_0^{-2} \\
\noalign{\vskip 2mm}
&&+ (p-2)\dfrac{\alpha Ky}{\mu \delta_0^{p-2}} 
+3(p-2)\dfrac{k \|u_0\|_{\infty}^{\alpha}Ky^{2-\gamma}}{ \delta_0^{p-2}} 
-(2p-3)\sigma \gamma  \biggr\}. \\
\end{array}
\end{equation}
Now, we may choose $\alpha =\alpha(p,q,\Omega,\mu,\|u_0\|_{C^2},\sigma)>1$ close enough to $1$ in such a way that  $\gamma=(\alpha-1)(1-2\sigma)$ is  small enough to satisfy
\begin{equation}
\gamma\left[\frac{q^2}{2\sigma}C_0^{q-p+1}+25(p-2)^2 \dfrac{ K^2\|u_0\|_{\infty}}{2\sigma\delta_0^{q-p}}\right]\leq
(q-p+1){ \alpha}
\end{equation}
and 
\begin{eqnarray}
&&(\alpha -1)\left(2\sigma -\dfrac{1}{q-p+1}\right)+\dfrac{2(q-p)}{q-p+1}\geq 0,\\
&&\alpha \dfrac{q-p}{q-p+1}-(\alpha-1)(1-2\sigma)\geq 0.
\end{eqnarray}
Finally, once $\alpha$ is fixed (hence $\gamma$ is also fixed small), recalling that $y\leq y_1\leq1$,  $x\leq  L_1$ and $\gamma\leq 2$,  
we take $k_0=k_0(p,q,\Omega,\mu,\|u_0\|_{C^2},\sigma)>0$ possibly smaller, in such a way that
\begin{equation}
\begin{array}{ll}
k_0 \left((q+2(p-2)) C_0^{\alpha+q-p+2} \delta_0^{-2} +\alpha(4p-6)\delta_0^{-2} C_0^{\alpha +1}\right)\\
+k_0^2 \big(4(p-2)(p-1)\gamma^2 L_1^2 C_0^{2\alpha} \delta_0^{-2}\big) 
+k_0 \dfrac{3(p-2) \|u_0\|_{\infty}^{\alpha}K}{ \delta_0^{p-2}} 
\leq\dfrac{2p-3}{2}\sigma \gamma,
\end{array}
\end{equation}
and next we take $y_1=y_1(p,q,\Omega,\mu,\|u_0\|_{C^2},\sigma)>0$ small enough such that
\begin{equation}
 \dfrac{ (p-2)\alpha Ky_1}{\mu \delta_0^{p-2}}\leq\dfrac{2p-3}{2}\sigma \gamma.
\end{equation}
Then it follows from \eqref{finalcalc} that
\begin{equation}
\mathcal{P}J\leq 0\qquad\text{in } D\times (T/2, T).
\end{equation}

Finally, we need to check (\ref{propAH}).
The continuity statement is clear from the definition of  $\mathcal{A}, \mathcal{H}$.
Let us show that $\mathcal{A}$ is bounded in $D\times (T/2, \tau)$ for each $\tau<T$. 
For this purpose, let us observe that due to $|\nabla u|\leq C(\tau), u\leq C(\tau)y$  and $\alpha-1\geq \gamma$, 
we have for $y\leq 1$ and $\tau\in(T/2,T)$
\begin{eqnarray}
|F'd|=\alpha u^{\alpha-1}y^{-\gamma}\leq C^{\alpha-1}(\tau)y^{\alpha-1-\gamma}\leq \alpha C^{\alpha-1}(\tau)\\
|Fd'|=\gamma u^{\alpha}y^{-\gamma-1}\leq \gamma C^{\alpha}(\tau)\\
|Fd|=u^{\alpha}y^{-\gamma}\leq C^{\alpha}(\tau).
\end{eqnarray}
We also have by (\ref{BernsteinEst}) and  (\ref{lowerNabla}):
\begin{equation}
|\nabla u|^r\leq\left\{
\begin{array}{ll}
C^r(\tau), &\qquad\text{if } r>0,\\
\delta_0^r,&\qquad\text{if } r<0.
\end{array}
\right.
\end{equation}
 The assertion then follows easily from (\ref{defA}), (\ref{defA1b}) and (\ref{uC21}). \end{proof}

	\section{APPENDIX 1. Proof of regularity results (Theorem \ref{theorsou2})}
 
\begin{proof}[Proof of Theorem \ref{theorsou2}(i)]
We assume $\delta_0:=\inf_{Q_T}|\nabla u|>0$.
Fix $0<\tau<T$ and let $M_\tau=\|\nabla u\|_{L^\infty(Q_\tau)}<\infty$.
We pick smooth functions $b=b_\tau$ and $F=F_\tau$ with the following properties:
$$\hbox{$b(s)=s^{(p-2)/2}$ \ and \ $F(s)=s^{q/2}$ \ \ for $\delta_0^2\le s\le M_\tau^2$,}$$
$$\inf_{[0,\infty)}b>0,\ \quad b'\ge 0,\ \quad b'(s)=0\ \hbox{ for $s$ large enough,}
\ \quad F\ge 0,\ \quad \sup_{[0,\infty)}F<\infty.
$$
By the results in \cite[Chapter V]{lady} (see Remark~\ref{remsou2} below for details),
there exists a (unique) classical solution $v=v_\tau\in C^{2+\alpha,1+\alpha/2}(\overline \Omega\times (0,\tau))\cap C(\overline Q_\tau)$,
for some $\alpha\in (0,1)$, of the problem
\begin{eqnarray*}
v_t-\nabla\cdot(b(|\nabla v|^2)\nabla v)&=&F(|\nabla v|^2) \quad\hbox{in $Q_\tau$}\\
v&=&g\quad\hbox{on $S_\tau$}\\
v(\cdot,0)&=&u_0\quad\hbox{in $\Omega$}.
\end{eqnarray*}
Since $v$ is also a weak solution of \eqref{1a}--\eqref{1c} in $Q_\tau$,
by uniqueness of weak solutions (cf.~Theorem \ref{theorsou1}(ii)),
it follows that $u=v_\tau$ on $Q_\tau$, hence (\ref{uC21}).
\end{proof}

\begin{rem}\label{remsou2} 
More precisely, in the special case when $u_0\in C^{2+\alpha}(\overline \Omega)$ and $u_0$ satisfies
the second order compatibility conditions, the existence of $v$ claimed in the above proof
follows from \cite[Theorem~V.6.1]{lady}.
In the general case $u_0 \in W^{1,\infty}(\Omega)$, with $u_0=g$ on $\partial\Omega$, 
this follows by a standard approximation procedure of $u_0$ by such smooth $u_{0,n}$.
Namely, if $v_n$ denotes the solution originating from $u_{0,n}$, then, by \cite[Theorems~V.4.1, V.1.1 and V.5.4]{lady} respectively,
we get uniform a priori estimates for the sequence $v_n$
in the spaces $L^\infty(0,\tau;W^{1,\infty}(\Omega))$, $C^\alpha(\overline Q_\tau)$ and 
$C^{2+\alpha,1+\alpha/2}_{loc}(\overline\Omega\times (0,\tau])$ for some $\alpha\in (0,1)$.
We may then pass to the limit along a subsequence and obtain a solution with the announced properties.
\end{rem}

In the proof of Theorem \ref{theorsou2}(ii)(iii), we shall use the following local regularity lemma.
We note that only statement (ii) will be used here.
The global version of statement (i) was already proved in Theorem \ref{theorsou2}(i).
However, we give and prove its local version for completeness, since
it was mentioned without proof in \cite[p.~2487]{AmalJDE}.

\begin{lem}\label{lemRegulLoc}
Under the assumptions of Theorem \ref{theorsou1}, let $u$ be the (maximal) weak solution of (\ref{eqprincipale})
and let $P_0=(x_0,y_0,t_0)\in Q_T$.
Assume $|\nabla u(P_0)|>0$. Then:
\smallskip

(i) for some $\alpha\in (0,1)$, $u$ is a classical $C^{2+\alpha,1+\alpha/2}$-solution on a space-time neighborhood of~$P_0$;

\smallskip
(ii) for some $\beta\in (0,1)$, $\nabla u$ is $C^{2+\beta,1+\beta/2}$ on a space-time neighborhood of~$P_0$.

\end{lem}

\begin{proof}
(i) Since, by Theorem~\ref{theorsou1}(iii), $\nabla u$ is continuous in $Q_T$, 
there exist $\lambda, \rho, M_2>0$ such that 
\begin{equation}\label{BoundLambda}
\lambda\le |\nabla u|\le M_2\quad\hbox{ in $Q^\rho:=B_\rho(x_0,y_0)\times [t_0-\rho,t_0+\rho]\subset Q_T$.}
\end{equation}
For any unit vector $\vec e$ and $0<h<\rho/2$, let us introduce the differential quotients 
$$D_hu=h^{-1}(\tau_hu-u),\quad\hbox{ where } \tau_hu=u\bigl((x,y)+h\vec e,t\bigr).$$
We have
$$|\nabla \tau_hu|^q-|\nabla u|^q=d^h(x,y,t)\cdot\nabla (\tau_hu-u)
 \ \ \hbox{ in $Q^{\rho/2}$},$$
where $|d^h(x,y,t)|\le C$ independent of $h$.
Next denote $b(s)=s^{(p-2)/2}$ and $a_i(p)=b(|p|^2)p_i$ where $p=(p_1,p_2)$, so that $\Delta_pu=\partial_i(a_i(\nabla u))$.
Following \cite[p.445]{lady}, we write
$$a_i(\nabla \tau_hu)-a_i(\nabla u)=\int_0^1\frac{d}{ds}\,a_i(s\nabla \tau_hu+(1-s)\nabla u))\, ds=
\tilde a_{ij}^h\partial_j(\tau_hu-u),$$
where
$$\tilde a_{ij}^h(x,y,t)=\int_0^1 \frac{\partial a_i}{\partial p_j}(s\nabla \tau_hu+(1-s)\nabla u))\, ds.$$
Subtracting the PDE in (\ref{eqprincipale}) for $u$ and for $\tau_hu$ and dividing by $h$, 
we see that $D_hu$ is a local weak solution of 
\begin{equation}\label{EqnTranslation}
\partial_t(D_hu)-\partial_i\bigl[\tilde a_{ij}^h\partial_j(D_hu)\bigr]=d^h(x,y,t)\cdot\nabla (D_hu)
 \ \ \hbox{ in $Q^{\rho/2}$}.
\end{equation}
Moreover, since $\frac{\partial a_i}{\partial p_j}\xi_i\xi_j=b(|p|^2)|\xi|^2+2b'(|p|^2)p_ip_j\xi_i\xi_j
\ge b(|p|^2)|\xi|^2\ge \lambda^{p-2}|\xi|^2$ in $Q^{\rho/2}$ by (\ref{BoundLambda}), we have
$$\tilde a_{ij}^h\xi_i\xi_j\ge \lambda^{p-2}|\xi|^2
 \ \ \hbox{ in $Q^{\rho/2}$}.$$
We then  test (\ref{EqnTranslation}) with $\varphi^2D_hu$, where $\varphi\in C^\infty_0(Q^{\rho/2})$ is a cut-off function such that $\varphi=1$ on $Q^{\rho/3}$.
By integration by parts and some simple manipulations, it is easy to see that
$$ \lambda^{p-2} \int_{Q^{\rho/3}} |\nabla D_hu|^2\, dxdydt \le
\int_{Q^{\rho/2}}
\tilde a_{ij}^h\partial_i(D_hu)\partial_j(D_hu)\varphi^2\, dxdydt\le C.$$
It follows that $D^2u\in L^2(Q^{\rho/3})$. Consequently, we obtain that $u\in W^{2,1}_2(Q^{\rho/3})$ and is a local strong solution of
equation (\ref{eqprincipale}) written in nondivergence form, i.e.:
\begin{equation}\label{nondiveq}
u_t-a_{ij}u_{ij}=f \ \hbox{ in $Q^{\rho/3}$},
\quad\hbox{where }
a_{ij}=|\nabla u|^{p-2}\Bigl[\delta_{ij}+(p-2)\dfrac{u_iu_j}{|\nabla u|^2}\Bigr],\quad f=|\nabla u|^q.
\end{equation}
Since, by Theorem~\ref{theorsou1}(iii), $a_{ij}, f$ are H\"older continuous in $\overline Q^{\rho/3}$,
it follows from interior Schauder parabolic regularity \cite[Theorem~III.12.2]{lady} that, for some $\alpha\in(0,1)$,
\begin{equation}\label{HolderReg}
u\in C^{2+\alpha,1+\alpha/2}(\overline Q^{\rho/4}).
\end{equation}

(ii) Thanks to (\ref{HolderReg}), we know that $u$ is a classical solution of (\ref{nondiveq}) in $Q^{\rho/4}$.
Keeping the above notation, for $0<h<\rho/8$, we then have
$$(D_hu)_t-a_{ij}(D_hu)_{ij}=F_h:=D_hf+(D_ha_{ij})(\tau_hu_{ij}) \quad\hbox{ in $Q^{\rho/8}$}.$$
Moreover, as a consequence of (\ref{HolderReg}), we have, for $1<r<\infty$, 
$$\|F_h\|_{L^r(Q^{\rho/8})}\le C\|\nabla f\|_{L^r(Q^{\rho/4})}+
\|\nabla A\|_{L^r(Q^{\rho/4})}\|D^2u\|_{L^\infty(Q^{\rho/4})}\le C, \quad 0<h<\rho/8.$$
It thus follows from interior parabolic $L^r$ estimates (see \cite[Theorem~III.12.2]{lady}) that, for $0<h<\rho/8$,
$$\|D^2 D_hu\|_{L^r(Q^{\rho/16})}+\|\partial_tD_h\|_{L^r(Q^{\rho/16})}\le C(\rho)
\bigl(\|F_h\|_{L^r(Q^{\rho/8})}+\|D_hu\|_{L^r(Q^{\rho/8})}\bigr)\le C.
$$
We deduce that $Du_t, D^3u\in L^r_{loc}(Q_T)$. Then differentiating (\ref{nondiveq}) in space, 
we see that the function $z=\partial_\ell u_x$ ($\ell=1,2$) is a local strong solution of
\begin{equation}\label{nondiveq2}
z_t-a_{ij}z_{ij}=\tilde f \quad\hbox{in $Q^{\rho/16}$,}
\end{equation}
where $\tilde f=\partial_\ell f-u_{ij}\partial_\ell a_{ij}$.
Since $a_{ij}, \tilde f$ are H\"older continuous in $\overline Q^{\rho/16}$ due to (\ref{HolderReg}),
it follows from interior Schauder parabolic regularity \cite[Theorem~III.12.2]{lady} that $z\in C^{2+\alpha,1+\alpha/2}(\overline Q^{\rho/20})$
for some $\alpha\in(0,1)$.
\end{proof}

\medskip

\begin{proof}[Proof of Theorem \ref{theorsou2} (continued)]
(ii) This is a direct consequence of Lemma~\ref{lemRegulLoc}.

\smallskip
(iii) It follows from (i)(ii) that $v=u_x\in C^{2,1}(Q_T)\cap C(\overline\Omega\times (0,T))$ is a classical solution of (\ref{nondiveq2}) in $Q_T$,
where $a_{ij}$ are defined in (\ref{nondiveq}).
Moreover, $v=g_x=0$ on $S_T$.
Taking $\theta(t)$ a cut-off in time and setting $w=\theta v$, we see that $w$ solves
\begin{equation}\label{nondiveq3}
w_t-a_{ij}w_{ij}=\bar f:=\theta \tilde f+\theta_t v \quad\hbox{in $Q_T$,}
\end{equation}
with $0$ initial-boundary conditions.
By \cite[Theorem~4.28]{lieb}, since $\bar f$ is locally H\"older continuous in $\overline\Omega\times [0,T)$ 
due to (i), there exists 
 a solution to this problem in $C^{2+\beta,1+\beta/2}(\overline\Omega\times [0,T))$ for some $\beta\in (0,1)$.
Since we have uniqueness in the class $C^{2,1}(Q_T)\cap C(\overline\Omega\times [0,T))$
by the maximum principle,
the conclusion (\ref{uxC21}) follows.
\end{proof}

\section{APPENDIX 2. A parabolic version of Serrin's corner lemma}

In \cite[p.~512]{single}, a Serrin corner property in a rectangle was shown for a parabolic equation involving the Laplacian.
This was proved by comparison with a suitable product of functions of $x,t$ and $y,t$.
This result and method are no longer sufficient here and we shall establish a result for general nondivergence operators 
by modifying the original proof of \cite{serrin} for the elliptic case.

\begin{prop}\label{CornerLemma}
Let $0< X_1<\hat  X_1, \ 0< Y_1<\hat  Y_1$,   $\hat D=(0,\hat  X_1)\times (0,\hat  Y_1)\subset \R^2$, $0<\tau_1<\tau_2$, $\hat D_{\tau}=\hat D\times (\tau_1,\tau_2)$.
Let the coefficients $a_{ij}=a_{ij}(x,y,t)$ satisfy 
\begin{equation}\label{mike1} 
a_{ij}\xi_i\xi_j\ge \lambda|\xi|^2\ \ \hbox{\, in $\overline{\hat D_{\tau}}$} \end{equation}
for some $\lambda>0$ and assume that
\begin{equation}\label{mikef}
a_{ij}, B_i \in C(\overline{\hat D_{\tau}}),
\quad a_{12}+a_{21}\ge - C(x\wedge y),  \quad B_1\ge -C x, \quad B_2\ge -C y \quad\text{in}\, \,\overline{\hat D_{\tau}}.
\end{equation}
Let $z\in C^{2,1}(\hat D_\tau)\cap C(\overline{\hat D_{\tau}})$ satisfy
\begin{equation}\label{LzNegative}
\mathcal{L}z:=z_t-a_{ij}z_{ij}-B_iz_i\leq 0 \ \ \text{in $ \hat D_{\tau}$},\quad z(x,y,t)\le 0\ \text{in $\overline{\hat D_{\tau}}$}, \quad  z(0,0,t)=0.
\end{equation}
Then, for each $t_0\in (\tau_1,\tau_2)$, there exists $c_0>0$ such that 
\begin{equation}\label{conclCorner}
z\leq -c_0 xy \qquad\text{in $(0, X_1)\times (0, Y_1)\times [t_0,\tau_2)$}.
\end{equation}

\end{prop}
\begin{proof}
Let $a=\min(X_1, Y_1,\frac{t_0-\tau_1}{2})$ and $\tau_3=\frac{\tau_1+t_0}{2}$, so that $\tau_1<\tau_3<t_0<\tau_2$.
 Fix $t_1\in [t_0, \tau_2)$ and let 
$$K_1:=\left\{(x,y,t);\, x^2+(a-y)^2+(t_1-t)^2<a^2,\ x>0, \ t\le t_1\right\}.$$
Observe that $K_1\subset  \hat D\times [\tau_3,t_1]$  
and set $K_2=B\bigl((0,0),a/2\bigr)\times [\tau_3,t_1]$ and $K=K_1\cap K_2$.

Now set 
$$\bar{v}(x,y,t):=e^{-\alpha(x^2+(y-a)^2+(t-t_1)^2)}-e^{-\alpha a^2},\quad
v(x,y,t)=e^{-\alpha(x^2+(y-a)^2+(t-t_1)^2)},$$
with $\alpha>0$ to be chosen later on, and define the auxiliary function $h=x \bar{v}$.
It is clear that $h>0$ in $K$. We compute 
$$h_t=-2\alpha x(t-t_1)v,\qquad
\nabla h=\begin{pmatrix}
\bar v-2\alpha x^2v\\ 
-2\alpha x(y-a)v
\end{pmatrix},
$$
$$D^2 h=v\begin{pmatrix}
-6\alpha x+4\alpha^2x^3&-2\alpha(y-a)+4\alpha^2x^2(y-a)\\
-2\alpha(y-a)+4\alpha^2x^2(y-a)&-2\alpha x+4\alpha^2x(y-a)^2
\end{pmatrix},
$$
\begin{equation*}
B(x,y,t)\cdot\nabla h
=-2\alpha xv[x B_1+(y-a)B_2]+ B_1\bar v.
\end{equation*}
Using that $a_{ij} \xi_i\xi_j\geq \lambda |\xi|^2$, we have 
\begin{align*}
a_{ij}h_{ij} 
&= v a_{11}(-6\alpha x+4\alpha^2x^3) +v a_{22}\bigl(-2\alpha x+4\alpha^2x(y-a)^2\bigr)\\
&\qquad+ v\, (a_{12}+a_{21})\bigl(-2\alpha(y-a)+4x^2\alpha^2(y-a)\bigr)\\
&=\alpha xv\biggl[4\alpha\bigl( a_{11} x^2+(y-a) x(a_{12}+a_{21})+(y-a)^2 a_{22}\bigr)\\
&\qquad -6a_{11}-2 a_{22}- \frac{2(y-a)(a_{12}+a_{21})}{x}\biggr]\\
&\ge \alpha xv\left[4\alpha  \lambda(x^2+(y-a)^2)
-6a_{11}-2 a_{22}+ \frac{2(a-y)(a_{12}+a_{21})}{x}\right],
\end{align*}
hence
\begin{align*}
\mathcal{L}h
&\le \alpha xv\left[-4\alpha  \lambda(x^2+(y-a)^2) +2(t_1-t)+6a_{11} +2 a_{22}- \frac{2(a-y)(a_{12}+a_{21})}{x}\right. \\
&\qquad \left. +2x B_1+2(y-a)B_2- \frac{B_1}{\alpha x}\bigl(1-e^{\alpha (x^2+(y-a)^2+(t-t_1)^2-a^2)}\bigr)\right].
\end{align*}
On the one hand, on $K$, we have $y< a/2$, hence $x^2+(y-a)^2 > a^2/4$. On the other hand,
using part of assumptions \eqref{mikef} along with $0\le a-y\le a$ and $0\le 1-e^{\alpha (x^2+(y-a)^2+(t-t_1)^2-a^2)}\le 1$ on $K$,
it follows that for $\alpha>1$ large enough,
\begin{align}
\mathcal{L}h
&\le  \alpha xv\left[-\alpha  \lambda a^2 +2(t_1-t)+6a_{11} +2 a_{22}+2Ca+2x B_1+2(y-a)B_2+C \right] &\notag \\
&\le  -\frac{\lambda\alpha^2a^2 xv}{2}<0\quad\hbox{ in $K$}. \label{LhNegative}
\end{align}

We now set $w=z+\eps h$ where $\eps$ is a positive constant to be chosen. By (\ref{LzNegative}) and (\ref{LhNegative}), we have
\begin{equation}\label{LwNegative}
\mathcal{L}w<0\quad\text{in $K$}.
\end{equation}
Denote $M=\max_{\overline K}w\ge 0$. Since $\mathcal{L}$ is (uniformly) parabolic, 
by the usual proof of the maximum principle, it follows from (\ref{LwNegative}) that $w$ cannot attain the value $M$ 
in $K$ (observe that  for each $s\in [\tau_3,t_1]$, the section $K\cap\{t=s\}$ is an open, possibly empty, subset of $\R^2$). To show $M=0$ (for sufficiently small $\eps>0$), it thus suffices to verify that 
$w\le 0$ on  $\partial_PK=\partial K\setminus (K\cap\{t=t_1\})$. We have $\partial_PK=\Gamma_1\cup\Gamma_2$, where
$\Gamma_1=\partial K_1\cap \overline K_2$  
and $\Gamma_2=\partial K_2\cap K_1$.  

On $\Gamma_1$ we have either 
$$x^2+(y-a)^2+(t_1-t)^2=a^2\quad\text{or}\quad x=0,$$
so that $h=0$ and $z\le 0$, 
hence $w\le 0$.
Next observe that on $\Gamma_2$ we have 
$$x^2+(y-a)^2+(t_1-t)^2<a^2 \quad\text{and}\quad x^2+y^2=a^2/4,$$
hence $\tau_1<\tau_3\le t\le t_1$ and $a/8<y<a/2$ (in other words, $(x,y)$ is ``far'' from the corners of $\hat D$).
Therefore, by the Hopf boundary point lemma \cite[Theorem~6~p.~174]{PW} and the strong maximum principle, 
there exists $c>0$ (independent of $t_1$), such that $z\le -cx$ on $\Gamma_2$.
Choosing $\eps\in (0,c)$, we then have
$w\leq -c x+\eps x<0$ on $\Gamma_2$.

We have thus proved that $M=0$ that is, $w\le 0$ in $K$. 
Letting $\tilde a:=a/(2\sqrt{2})$ and noting that $\{0<x\le y<\tilde a\}\times\{t_1\}\subset K$, we get
\begin{align*}
z(x,y,t_1)
&\le -\eps h(x,y,t_1)=-\eps x e^{-\alpha a^2}\bigl(e^{\alpha(a^2-x^2-(y-a)^2)}-1\bigr) \\
&\le -\eps \alpha e^{-\alpha a^2}x \bigl(a^2-x^2-(y-a)^2\bigr)
=-\eps \alpha e^{-\alpha a^2} x \bigl(2ay-x^2-y^2\bigr)\\
&\le -a\eps \alpha e^{-\alpha a^2} xy\quad\hbox{ for $0<x\le y<\tilde a$.}
\end{align*}

Now exchanging the roles of $x,y$ and noticing that the assumptions \eqref{mikef} are symmetric in $x,y$,
the conclusion already obtained guarantees that also
$z(x,y,t_1)\le -a\eps \alpha e^{-\alpha a^2} xy$ for $0<y\le x<\tilde a$,
hence \eqref{conclCorner} in $(0,\tilde a)^2\times [t_0,\tau_2)$.
The extension to the remaining part of the rectangle $(0, X_1)\times (0, Y_1)$
(away from the corner $(0,0)$) follows from the Hopf boundary lemma and the strong maximum principle.
\end{proof}

\noindent
\vspace{0.2cm} \\
Universit\'e Paris 13, Sorbonne Paris Cit\'e, \\
Laboratoire Analyse, G\'eom\'etrie et Applications, CNRS (UMR 7539), \\
93430 Villetaneuse, France. \\
E-mail: attouchi@math.univ-paris13.fr, souplet@math.univ-paris13.fr

\end{document}